\documentclass[12 pt]{article}%
\usepackage{amsmath, amsfonts, amsthm, color,latexsym}
\usepackage{amsmath}
\usepackage{amsfonts}
\usepackage{amssymb}
\usepackage{color, soul}
\usepackage[all]{xy}
\usepackage{graphicx}%
\providecommand{\U}[1]{\protect\rule{.1in}{.1in}}
\allowdisplaybreaks[4]
\newtheorem{theorem}{Theorem}[section]
\newtheorem{proposition}[theorem]{Proposition}

\newtheorem{example}[theorem]{Example}

\newtheorem{lemma}[theorem]{Lemma}
\newtheorem{final remark}[theorem]{Final Remark}
\newtheorem{definition}[theorem]{Definition}
\allowdisplaybreaks[4]
\begin{document}

\title{On the representation of linear functionals on hyper-ideals of multilinear operators}
\author{Geraldo Botelho\thanks{Supported by CNPq Grant
304262/2018-8 and Fapemig Grant PPM-00450-17.}\,\, and Raquel Wood\thanks{Supported by a CAPES scholarship.\newline 2020 Mathematics Subject Classification:  47H60, 47A67,  46G25, 47L22, 47B10.\newline Keywords: Operator ideals, multilinear operators, Borel transform, hyper-ideals.
}}
\date{}
\maketitle

\begin{abstract} A standard technique in infinite dimensional holomorphy, which produced several useful results, uses the Borel transform to represent linear functionals on certain spaces of multilinear operators between Banach spaces as multilinear operators. In this paper we develop a technique to represent linear functionals, as linear operators, on spaces of multilinear operators that are beyond the scope of the standard technique. Concrete applications to some well studied classes of multilinear operators, including the class of compact multilinear operators, and to one new class %of hyper-$\sigma(p)$-nuclear operators
 are provided. We can see, in particular, that sometimes our representations hold under conditions less restrictive than those of the related classical ones.
\end{abstract}

\section{Introduction}
The representation of linear functionals is a classical and very useful topic in Functional Analysis and Operator Theory, going back to the celebrated representation theorems due to F. Riesz. Starting with the representation of functionals on sequence spaces and function spaces we all learned at graduate school, the subject spread through all subareas of mathematical analysis. In Infinite Dimensional Holomorphy, the representation of linear functionals on spaces of multilinear operators, polynomials and holomorphic functions has been a valuable tool since the 1960's with the seminal works of Gupta \cite{gupta1, gupta2}, who introduced the use of the Borel transform in the area. Several results followed, for example Dwyer's and Carando-Dimant's representations of functionals on spaces of nuclear multilinear operators/polynomials \cite{dwyer, cardim}, Alencar's representation of linear functionals on spaces of multilinear operators/polynomials that can be approximated by finite type operators \cite{alencar} and Matos' representation of functionals on spaces of $(s;r_1, \ldots,r_n)$-nuclear multilinear operators \cite{matos}. The books \cite{dineen1, dineen2} by Dineen are excellent sources for results of this kind and their many applications. Recent developments, including applications to linear dynamics, can be found, e.g., in \cite{ximenaLAA, boyd, vinari, vinmu1, vinmu2, vindan, muro}.

Let us give a brief description of the Borel transform technique. By $E^*$ we denote the (topological) dual of the Banach space $E$. Given a subspace ${\cal M}(E_1, \ldots, E_n;F)$ of the space of $n$-linear operators from $E_1 \times \cdots \times E_n$ to $F$, where $E_1, \ldots, E_n,F$ are Banach spaces, endowed with a complete norm $\|\cdot\|_{\cal M}$, the idea is to identify a subspace ${\cal R}(E_1^*, \ldots, E_n^*;F^*)$ of the space of $n$-linear operators from $E_1^* \times \cdots \times E_n^*$ to $F^*$  and a complete norm $\|\cdot\|_{\cal R}$ on it such that the Borel transform
$$\beta_n \colon ({\cal M}(E_1, \ldots, E_n;F),\|\cdot\|_{\cal M})^* \longrightarrow ({\cal R}(E_1^*, \ldots, E_n^*;F^*), \|\cdot\|_{\cal R})$$
\vspace*{-1.7em}
$$\beta_n(\phi)(x_1^*, \ldots, x_n^*)(y) = \phi(x_1^*\otimes \cdots \otimes x_n^*\otimes y),  $$
is an isometric isomorphism. By $x_1^*\otimes \cdots \otimes x_n^*\otimes y$ we mean the $n$-linear operator given by
$$(x_1^*\otimes \cdots \otimes x_n^*\otimes y)(x_1, \ldots, x_n) = x_1^*(x_1)\cdots x_n^*(x_n)y. $$
Although very fruitful, this technique works only for spaces ${\cal M}(E_1, \ldots, E_n;F)$ contained in the space of multilinear operators that can be approximated, in the usual operator norm, by operators of finite type (linear combinations of operators of the form $x_1^*\otimes \cdots \otimes x_n^*\otimes y$). The purpose of this paper is to introduce a new technique, called the {\it hyper-Borel transform}, that produces representations of functionals on spaces of multilinear operators larger than the ones encompassed by the standard technique. In Section 2 we give the basics of the hyper-Borel transform and show how a representation of functionals on spaces of linear operators can be passed to a representation of functionals on composition hyper-ideals of multilinear operators. In Section 3 we apply the hyper-Borel transform to, as far as we know for the first time, represent functionals on the well studied classes of approximable, compact, hyper-nuclear and hyper-$(s,r)$-nuclear multilinear operators and to the new class of hyper-$\sigma(p)$-nuclear multilinear operators. %We point out that some of these representations hold under conditions less restrictive than those of the related classical representations %As far as we know, all these representations are original, in particular the case like all the other cases, this is first representation theorem for

It is worth mentioning that the standard and the new techniques are mutually independent in the sense that each representation obtained using one of the techniques cannot be obtained using the other one. The reason is that, on the one hand, the Borel transform works precisely for Banach spaces of multilinear operators in which the operators of finite type are dense. On the other hand, we prove in Section 2 that the hyper-Borel transform works precisely for Banach spaces of multilinear operators in which the finite rank operators are dense -- recall that a vector-valued function has finite rank if its range generates a finite dimensional subspace of the target space. Nevertheless, some of the representations we prove in Section 3 are closely related to classical representations. We point out that sometimes our representations hold under conditions less restrictive than those of the related classical ones. %is neither more nor less general than a result obtained with the other technique.

Multi-ideals of multilinear operators, in the sense of \cite{fg,fh}, contain the finite type operators, and this is the reason why the Borel transform technique is usually applied to represent functionals on multi-ideals. Hyper-ideals of multilinear operators, in the sense of \cite{ewertonlaa, velanga}, contain the finite rank operators, and this is the reason why we shall apply the hyper-Borel transform to represent functionals on hyper-ideals.

By $\mathcal{L}(E_{1},\ldots,E_{n};F)$ we denote the Banach space of $n$-linear continuous operators from $E_1 \times \cdots \times E_n$ to $F$. In the linear case $n = 1$ e simply write ${\cal L}(E;F)$ and the scalar-valued case where $F$ is the scalar field $\mathbb{K} = \mathbb{R}$ or $\mathbb{C}$ we write ${\cal L}(E_1, \ldots, E_n)$. The closed unit ball of $E$ is denoted by $B_E$. Given $A \in {\cal L}(E_1, \ldots, E_n)$ and $y \in F$ we define $A \otimes y \in {\cal L}(E_1, \ldots, E_n;F)$ by $A\otimes y (x_1, \ldots, x_n) = A(x_1, \ldots, x_n)y$. Multilinear operators of finite rank are linear combinations of operators of this kind. For the theory of (spaces of) multilinear operators and some eventual non explained notation we refer to \cite{dineen2, mujica}. Banach ideals of linear operators are taken in the sense of \cite{df, livropietsch}.

\section{The hyper-Borel transform}

In this section, $n\in \mathbb{N}$, $E_1, \ldots, E_n$ and $F$ are Banach spaces.

\begin{proposition}\label{firstpro}
Let $\mathcal{H}(E_{1},\ldots,E_{n};F)$ be a linear subspace of $\mathcal{L}(E_{1},\ldots,E_{n};F)$ endowed with a complete norm  $\|\cdot\|_{\mathcal{H}}$, containing the $n$-linear operators of finite rank and such that $\|A \otimes y\|_{\mathcal{H}} \leq \|A\| \cdot \|y\|$ for all $A \in \mathcal{L}(E_{1},\ldots,E_{n})$ and $ y \in F.$ Then:\\
{\rm (a)} The map
$${\cal B}_n\colon (\mathcal{H}(E_{1},\ldots,E_{n};F) , \|\cdot\|_{\mathcal{H}})^* \longrightarrow \mathcal{L}(\mathcal{L}(E_{1},\ldots,E_{n});F^*)~,~{\cal B}_n(\phi)(A)(y) = \phi(A \otimes y),$$
is a well defined continuous linear operator and $\|{\cal B}_n\| \leq 1$.\\
{\rm (b)} ${\cal B}_n$ is injective if and only if the subspace of finite rank operators is $\|\cdot\|_{\cal H}$-dense in $\mathcal{H}(E_{1},\ldots,E_{n};F)$.
\end{proposition}

\begin{proof} (a) Let us see that ${\cal B}_n$ is well defined. For $A \in \mathcal{L}(E_{1},\ldots,E_{n})$ and $y \in F$, $A \otimes y$ is a rank 1 operator, so it belongs to $\mathcal{H}(E_{1},\ldots,E_{n};F)$ and $\phi$ can be evaluated on it. It is plain that ${\cal B}_n(\phi)(A)$ is linear and the inequality $\|A \otimes y\|_{\mathcal{H}} \leq \|A\| \cdot \|y\|$ shows that it belongs to $F^*$. The linearity of ${\cal B}_n(\phi)$ is also obvious and its continuity follows from
\begin{align*}
\|{\cal B}_n(\phi)(A)\|& = \sup\limits_{y \in B_{F}} |{\cal B}_n(\phi)(A)(y)| = \sup\limits_{y \in B_{F}} |\phi(A \otimes y)|\\
&\leq \sup\limits_{y \in B_{F}} \|\phi\|_{\mathcal{H}^*} \cdot \|A \otimes y\|_{\mathcal{H}} = \|\phi\|_{\mathcal{H}^*} \cdot \sup\limits_{y \in B_{F}} \|A \otimes y \|_{\mathcal{H}}\\
&\leq \|\phi\|_{\mathcal{H}^*} \cdot \sup\limits_{y \in B_{F}} \|A\| \cdot \|y\| = \|\phi\|_{\mathcal{H}^*} \cdot \|A\|.
\end{align*}
Since the linearity of ${\cal B}_n$ is clear, the inequality above gives $\|{\cal B}_n(\phi)\| \leq \|\phi\|_{{\cal H}^*}$, which proves the remaining statements of (a).\\
(b) Suppose that the finite rank operators are $\|\cdot\|_{\cal H}$-dense in $\mathcal{H}(E_{1},\ldots,E_{n};F)$ and let $\phi \in {\rm ker}({\cal B}_n)$. Then
$$\phi\left(\sum_{j=1}^k A_j \otimes y_j \right) = \sum_{j=1}^k \phi(A_j \otimes y_j) = \sum_{j=1}^k {\cal B}_n(\phi)(A_j)(y_j) = 0$$
for all $A \in {\cal L}(E_1, \ldots, E_n)$ and $y \in F$, meaning that $\phi$ vanishes in the dense subspace of finite rank operators. The continuity of $\phi$ implies that  $\phi = 0$.

Assume now that ${\cal B}_n$ is injective and let $\phi \in \mathcal{H}(E_{1},\ldots,E_{n};F)^*$ be a functional that vanishes in the finite rank operators. So,
$${\cal B}_n(\phi)(A)(y) = \phi(A \otimes y) = 0$$
for all $A \in {\cal L}(E_1, \ldots, E_n)$ and $y \in F$, meaning that ${\cal B}_n(\phi)(A) = 0$ for every $A \in {\cal L}(E_1, \ldots, E_n)$. The injectivity of ${\cal B}_n$ gives $\phi = 0$. This proves that the only functional that vanishes on the subspace of finite rank operators is the null functional. This is enough to conclude, by the Hahn-Banach Theorem, that this subspace is $\|\cdot\|_{\cal H}$-dense in $\mathcal{H}(E_{1},\ldots,E_{n};F)$.
\end{proof}

In the linear case we have $\beta_1 = {\cal B}_1 =: {\cal B}$, which shall be referred to as the linear Borel transform.

\begin{example}\rm Let $({\cal H}, \|\cdot\|_{\cal H}) $ be a Banach hyper-ideal of multilinear operators. Then $(\mathcal{H}(E_{1},\ldots,E_{n};F),\|\cdot\|_{\cal H})$ satisfies the conditions of the proposition above for all $n$, $E_1, \ldots,$ $E_n, F$ with $\|A \otimes y\|_{\mathcal{H}} = \|A\| \cdot \|y\|$.
\end{example}

To our purpose of representing linear functions on spaces of multilinear operators, the very first question is when ${\cal B}_n$ is an isomorphism into with respect to the usual operator norm. To settle this question we have to recall the class of hyper-nuclear operators \cite{ewertonlaa}: an operator $A \in \mathcal{L}(E_{1},\ldots,E_{n};F)$ is {\it hyper-nuclear} if there are sequences $(\lambda_{j})_{j=1}^{\infty} \in \ell_{1}$, $(B_{j})_{j=1}^{\infty} \in \ell_{\infty}(\mathcal{L}(E_{1},\ldots,E_{n}))$ and $(y_{j})_{j=1}^{\infty} \in \ell_{\infty}(F)$ such that, for all $x_j \in E_j, j = 1,\ldots, n$,
$$A(x_1, \ldots, x_n) = \sum_{j=1}^\infty \lambda_jB(x_1, \ldots, x_n)y_j. $$
In this case we say that $A = \sum\limits_{j=1}^\infty \lambda_j B_j \otimes y_j$ is a hyper-nuclear representation of $A$. The hyper-nuclear norm of $A$ is defined as
$$\|A\|_{\cal HN}= \inf\sum\limits_{j=1}^{\infty} |\lambda_{j}| \cdot \|B_{j}\| \cdot \|y_{j}\| = \inf \|(\lambda_j)_{j=1}^\infty\|_1\cdot \|(B_j)_{j=1}^\infty\|_\infty \cdot \|(y_j)_{j=1}^\infty\|_\infty,$$
where the infima run over all hyper-nuclear representations of $A$. The class $(\mathcal{HN}, \|\cdot\|_{\mathcal{HN}})$ of hyper-nuclear multilinear operators is a Banach hyper-ideal.

In Theorem \ref{repHN} we shall prove that, in the presence of the approximation property, $[{\cal HN}(E_1, \ldots, E_n;F)]^*$ is isometrically isomorphic to $\mathcal{L}(\mathcal{L}(E_{1},\ldots,E_{n});F^*)$ by means of the hyper-Borel transform ${\cal B}_n$. For the moment we establish that  the only class for which the hyper-Borel transform can be an isomorphism into with respect to the usual operator norm is the class of hyper-nuclear operators.

By $E_1\widehat{\otimes}_{\pi} \cdots \widehat{\otimes}_{\pi} E_{n}$ we denote the completed projective tensor product of $E_1, \ldots, E_n$ and by $A_L$ the linearization of a multilinear operator $A \in {\cal L}(E_1, \ldots, E_n;F)$, that is, $A_L \in \mathcal{L}(E_{1} \widehat{\otimes}_{\pi} \cdots \widehat{\otimes}_{\pi} E_{n};F)$ and $A_L(x_1 \otimes \cdots \otimes x_n) = A(x_1, \ldots, x_n)$ for all $x_j \in E_j$, $j = 1, \ldots, n$  (see \cite{ryan}).

\begin{theorem} Let $\mathcal{H}(E_{1},\ldots,E_{n};F)$ and ${\cal B}_n$ be as in Proposition {\rm \ref{firstpro}} and suppose that $\|\cdot\| \leq \|\cdot\|_{\mathcal{H}}$. If ${\cal B}_n$ is an isomorphism into, then $\mathcal{H}(E_{1},\ldots,E_{n};F) = {\cal HN}(E_1, \ldots, E_n;F)$,  ${\cal B}_n$ is onto $\mathcal{L}(\mathcal{L}(E_{1},\ldots,E_{n});F^*)$ and the norms $\|\cdot\|_{\mathcal{H}}$ and $\|\cdot\|_{\mathcal{HN}}$ are equivalent.
\end{theorem}

The condition $\|\cdot\| \leq \|\cdot\|_{\mathcal{H}}$ is a very natural one, for example, multi-ideals, in particular hyper-ideals, enjoy it.

\begin{proof} %Vejamos que $\mathcal{L}_{\mathcal{HN}} \subseteq \mathcal{H}$:
   Given $A \in \mathcal{HN}(E_{1},\ldots,E_{n};F)$ and $\varepsilon > 0$, let %. Então podemos tomar uma representação hiper-nuclear de $A$
$A = \sum\limits_{j=1}^{\infty} \lambda_{j} B_{j} \otimes y_{j}$ be a hyper-nuclear representation of $A$ such that
%,
%\end{equation*}
%$(\lambda_{j})_{j=1}^{\infty} \in \ell_{1}$, $(B_{j})_{j=1}^{\infty} \in \ell_{\infty}(\mathcal{L}(E_{1},\ldots,E_{n}))$ e $(y_{j})_{j=1}^{\infty} \in \ell_{\infty}(F)$, tal que
\begin{equation*}
\sum\limits_{j=1}^{\infty} |\lambda_{j}| \cdot \|B_{j}\| \cdot \|y_{j}\| \leq (1 + \varepsilon)\|A\|_{\mathcal{HN}}.
\end{equation*}
%Como
%\begin{equation*}
%\|A\|_{\mathcal{HN}} = \inf \{ \sum\limits_{j=1}^{\infty} |\lambda_{j}| \cdot \|B_{j}\| \cdot \|y_{j}\| \}
%\end{equation*}
%temos que $\sum\limits_{j=1}^{\infty} |\lambda_{j}| \cdot \|B_{j}\| \cdot \|y_{j}\| < \infty$. Consideremos
For $m \in \mathbb{N}$, set $A_{m} := \sum\limits_{j=1}^{m} \lambda_{j} B_{j} \otimes y_{j} \in \mathcal{H}(E_{1},\ldots,E_{n};F)$.
%\medskip
%\noindent
%De $\sum\limits_{j=1}^{\infty} |\lambda_{j}| \cdot \|B_{j}\| \cdot \|y_{j}\| < \infty$, segue que dado $\varepsilon > 0$,
Let $m_{0} \in \mathbb{N}$ be such that $\sum\limits_{j=k+1}^{\infty} |\lambda_{j}| \cdot \|B_{j}\| \cdot \|y_{j}\| < \varepsilon$ for every  $k\geq m_{0}$. So,
\begin{align*}
\|A_{m} - A_{k}\|_{\mathcal{H}}& = %\left\| \sum\limits_{j=1}^{m} \lambda_{j} B_{j} \otimes y_{j} - \sum\limits_{j=1}^{k} \lambda_{j} B_{j} \otimes y_{j}\|_{\mathcal{H}}\\
  \left\|\sum\limits_{j=k+1}^{m} \lambda_{j} B_{j} \otimes y_{j}\right\|_{\mathcal{H}} \leq \sum\limits_{j=k+1}^{m} \|\lambda_{j}B_{j} \otimes y_{j}\|_{\mathcal{H}} \leq \sum\limits_{j=k+1}^{m} |\lambda_{j}| \cdot \|B_{j}\| \cdot \|y_{j}\| < \varepsilon
\end{align*}
for $m > k \geq m_{0}$, showing that $(A_{m})_{m=1}^\infty$ is a Cauchy sequence in   $\mathcal{H}(E_{1},\ldots,E_{n};F)$. Let $A_{1} \in \mathcal{H}(E_{1},\ldots,E_{n};F)$ be such that $A_{m} \longrightarrow A_{1}$ in $\mathcal{H}(E_{1},\ldots,E_{n};F)$.   Since $\|\cdot\| \leq \|\cdot\|_{\mathcal{H}}$ we have that $A_{m} \longrightarrow A_{1}$ in $\mathcal{L}(E_{1},\ldots,E_{n};F)$. On the other hand,
%\medskip
%\noindent
%Agora, notemos que
\begin{align*}
\|A_{m} - A\|&  %\left \| \sum\limits_{j=m+1}^{\infty} \lambda_{j} B_{j} \otimes y_{j} \right \|\\
 = \sup\limits_{\|x_{j}\|\leq 1} \left \|\sum\limits_{j=m+1}^{\infty} \lambda_{j} B_{j}(x_{1},\ldots,x_{n})y_j \right\| \\
& \leq \sup\limits_{\|x_{j}\|\leq 1} \sum\limits_{j=m+1}^{\infty} |\lambda_{j}| \cdot \|B_{j}\| \cdot \|y_{j}\| \cdot \|x_{1}\| \cdots \|x_{n}\|\\&
 = \sum\limits_{j=m+1}^{\infty} |\lambda_{j}| \cdot \|B_{j}\| \cdot \|y_{j}\| \stackrel{m \to \infty}{\longrightarrow}0,
\end{align*}
showing that $A_{m} \longrightarrow A$ in $\mathcal{L}(E_{1},\ldots,E_{n};F)$, %ou seja, $\|A_{m} - A\| \rightarrow 0$. Portanto
hence $A = A_{1} \in \mathcal{H}(E_{1},\ldots,E_{n};F)$. Moreover,
\begin{align*}
\|A\|_{\mathcal{H}}& = \|A_{1}\|_{\mathcal{H}} = \lim\limits_{m} \|A_{m}\|_{\mathcal{H}} = \lim\limits_{m} \left\| \sum\limits_{j=1}^{m} \lambda_{j}B_{j} \otimes y_{j} \right \|_{\cal H} \\
& \leq \lim\limits_{m}\sum\limits_{j=1}^{m} |\lambda_{j}| \cdot \|B_{j}\| \cdot \|y_{j}\| = \sum\limits_{j=1}^{\infty} |\lambda_{j}| \cdot \|B_{j}\| \cdot \|y_{j}\| \leq (1 + \varepsilon) \|A\|_{\mathcal{HN}}.
\end{align*}
%\medskip
%\noindent
 Making $\varepsilon \rightarrow 0$ we get $\|A\|_{\mathcal{H}} \leq \|A\|_{\mathcal{HN}}$. In particular, the inclusion $i \colon  \mathcal{HN}(E_{1},\ldots,E_{n};F) \hookrightarrow \mathcal{H}(E_{1},\ldots,E_{n};F)$ is a bounded linear operator.

 Now we start proving the reverse inclusion. The assumptions give that
%medskip
%\noindent
%Vejamos que $\mathcal{H} \subseteq \mathcal{L}_{\mathcal{HN}}:$
%\medskip
%\noindent
%$\bullet$ Tome
$$T \colon \mathcal{L}(E_{1},\ldots,E_{n}) \times F \longrightarrow {\mathcal{HN}}(E_{1},\ldots,E_{n};F)~,~T(B,y) = B \otimes y,$$
is a continuous bilinear operator. Consider the linearization
%Então, $T$ é bilinear e contínua. Pela Propriedade Universal do Produto Tensorial, existe
$$T_{L}\colon \mathcal{L}(E_{1},\ldots,E_{n}) \widehat{\otimes}_{\pi} F \longrightarrow {\mathcal{HN}}(E_{1},\ldots,E_{n};F)~, T_L(B \otimes y) = B \otimes y,$$
of $T$ %So, the adjoint of $i \circ T_L$,
%$$(i\circ T_L)^* = T_L^* \circ i^*: (\mathcal{H}(E_{1},\ldots,E_{n};F))^* \longrightarrow (\mathcal{L}(E_{1},\ldots,E_{n}) \widehat{\otimes}_{\pi} F)^*,$$
%is a bounded linear operator.
and the canonical isometric isomorphisms
%\newline
%$ \phi \mapsto u'(\phi)(B \otimes y) = \phi(u(B \otimes y)) = \phi(T(B,y)).$
% Again by Ryan, the operator
%\medskip
%\noindent
%Consideremos as aplicações
$$V_{1}\colon (\mathcal{L}(E_{1},\ldots,E_{n}) \widehat{\otimes}_{\pi} F)^* \longrightarrow \mathcal{L}(\mathcal{L}(E_{1},\ldots,E_{n}),F;\mathbb{K})~,~V_1(\varphi)(B,y) = \varphi(B \otimes y), {\rm ~and}$$
$$V_{2}\colon \mathcal{L}(\mathcal{L}(E_{1},\ldots,E_{n}),F;\mathbb{K}) \longrightarrow \mathcal{L}(\mathcal{L}(E_{1},\ldots,E_{n});F^*)~,~
V_{2}(\xi)(B)(y) = \xi(B,y),$$
(see \cite{mujica, ryan}) and note that $(V_2 \circ V_1)(\varphi)(B)(y) = \varphi(B \otimes y)$. So,
$$V_2 \circ V_1 \circ T_L^* \circ i^* \colon \mathcal{H}(E_{1},\ldots,E_{n};F)^* \longrightarrow \mathcal{L}(\mathcal{L}(E_{1},\ldots,E_{n});F^*)$$
is a bounded linear operator and, for all $\phi \in (\mathcal{H}(E_{1},\ldots,E_{n};F))^*$, $B \in \mathcal{L}(E_{1},\ldots,E_{n})$ and $ y \in F$,
\begin{align*}(V_2 \circ V_1 \circ T_L^* \circ i^* )(\phi)(B)(y) &= (V_2 \circ V_1)((i\circ T_L)^*(\phi))(B)(y)=  (i\circ T_L)^*(\phi)(B \otimes y) \\
 &= \phi(i(T_L))(B\otimes y) = \phi(B \otimes y) = {\cal B}_n(\phi)(B)(y).
 \end{align*}
 This shows that $ V_2 \circ V_1 \circ T_L^* \circ i^* = {\cal B}_n$ is an isomorphism into by assumption, in particular it is injective. Since $ V_2 \circ V_1$ is an isomorphism, we conclude that $ T_L^* \circ i^*$ is injective. Calling $\psi$ the restriction of the isomorphism $(V_2 \circ V_1)^{-1}$ to the range of ${\cal B}_n$, we have that
$$\psi \colon  {\cal B}_n(\mathcal{H}(E_{1},\ldots,E_{n};F)^*)  \longrightarrow (V_2 \circ V_1)^{-1}( {\cal B}_n(\mathcal{H}(E_{1},\ldots,E_{n};F)^*)$$
%$$ \psi({\cal B}_n(\varphi)) = (V_2 \circ V_1)^{-1}({\cal B}_n(\varphi)),
%$$
 is an isomorphism as well. We know that ${\cal B}_n(\mathcal{H}(E_{1},\ldots,E_{n};F)^*)$ is a Banach space because ${\cal B}_n$ is an isomorphism into, so $\psi({\cal B}_n(\mathcal{H}(E_{1},\ldots,E_{n};F)^*))$ is a Banach space too, hence a closed subspace of $(\mathcal{L}(E_{1},\ldots,E_{n}) \widehat{\otimes}_{\pi} F)^*$. For every $\phi \in (\mathcal{H}(E_{1},\ldots,E_{n};F))^*$,
 $$(T_L^* \circ i^*)(\phi) = (V_2 \circ V_1)^{-1} \circ {\cal B}_n (\phi) = \psi({\cal B}_n(\phi)), $$
 thus
 $$(i \circ T_L)^*(\mathcal{H}(E_{1},\ldots,E_{n};F)^*) = \psi({\cal B}_n(\mathcal{H}(E_{1},\ldots,E_{n};F)^*))$$ is a closed subspace of $(\mathcal{L}(E_{1},\ldots,E_{n}) \widehat{\otimes}_{\pi} F)^*$. Since $(i \circ T_L)^*$ is injective, it is an isomorphism into, therefore $i \circ T_L$ is surjective by \cite[Theorem 2.1]{carothers}. Therefore, $i$ is surjective, proving that $\mathcal{HN}(E_{1},\ldots,E_{n};F) = \mathcal{H}(E_{1},\ldots,E_{n};F)$. We have already proved that $i$ is continuous, so the Open Mapping Theorem provides the equivalence of the norms.

We just sketch the proof of the surjectivity of ${\cal B}_n$. Given $v \in \mathcal{L}(\mathcal{L}(E_{1},\ldots,E_{n});F^*)$, call $A_L \colon \mathcal{L}(E_{1},\ldots,E_{n}) \widehat{\otimes}_{\pi} F \longrightarrow \mathbb{K}$ the linearization of the continuous bilinear operator
$$A\colon \mathcal{L}(E_{1},\ldots,E_{n}) \times F \longrightarrow \mathbb{K}~,~A(B,y) = v(B)(y).$$
Identifying the space of finite rank multilinear operators with $\mathcal{L}(E_{1},\ldots,E_{n})\otimes F $ in the obvious way, we can consider the restriction $\varphi$ of $A_L$ to $\mathcal{L}(E_{1},\ldots,E_{n})\otimes F \subset \mathcal{HN}(E_{1},\ldots,E_{n};F)$. It is easy to see that $\varphi \in (\mathcal{L}(E_{1},\ldots,E_{n})\otimes F, \|\cdot\|_{\cal HN})^*$. In the beginning of the proof we showed that the partial sums of a hyper-nuclear representation converge in the norm $\|\cdot\|_{\cal HN}$ to the hyper-nuclear operator, which gives that $\mathcal{L}(E_{1},\ldots,E_{n})\otimes F $ is dense in $\mathcal{HN}(E_{1},\ldots,E_{n};F)$. Let $\widetilde{\varphi}$ be (unique) bounded linear extension of $\varphi$ to $\mathcal{HN}(E_{1},\ldots,E_{n};F)$. Then, for all $B \in \mathcal{L}(E_{1},\ldots,E_{n})$ and $y \in F$,
$${\cal B}_n(\widetilde{\varphi})(B)(y) = \widetilde{\varphi}(B \otimes y) = \varphi(B \otimes y) = A_{L}(B \otimes y) = A(B,y) = v(B)(y),$$
that is, ${\cal B}_n(\widetilde{\varphi}) = v$.
\end{proof}

Now it is clear that, for any class $\cal H$ different from the hyper-nuclear operators, to represent linear functionals on $\mathcal{H}(E_{1},\ldots,E_{n};F)$ it is necessary to replace the usual norm on ${\cal B}_n\hspace*{-0.2em}\left(\mathcal{H}(E_{1},\ldots,E_{n};F)^* \right) \subseteq \mathcal{L}(\mathcal{L}(E_{1},\ldots,E_{n});F^*)$ with some suitable norm. The obvious guess is to consider Banach operator ideal norms, but the following more general notion shall prove later to be more appropriate.

\begin{definition}\rm A {\it left semi-operator ideal} is a correspondence $\alpha$ that to each pair  $(E,F)$ of Banach spaces assigns a linear subspace $\mathcal{L}_{\alpha}(E;F^*)$ of $\mathcal{L}(E;F^*)$ endowed with a complete norm $\|\cdot\|_{\alpha}$ satisfying the left ideal property: if $v \in \mathcal{L}(E;F)$ and $u \in \mathcal{L}_{\alpha}(F;G^*)$, then $u \circ v \in \mathcal{L}_{\alpha}(E;G^*)$ and
\begin{equation}\label{semi}\|u \circ v \|_{\alpha} \leq \|u\|_{\alpha} \cdot \|v\|.
\end{equation}
The symbol $\mathcal{L}_{\alpha}(E;F^*)$ shall henceforth mean the Banach space $(\mathcal{L}_{\alpha}(E;F^*),\|\cdot\|_\alpha)$.
\end{definition}

\begin{example}\label{exclass}\rm (a) Every Banach operator ideal is a left semi-operator ideal.\\
(b) Let us see a left semi-operator ideal that is not an operator ideal and shall be useful later. Consider the following class introduced in \cite{ximenaLAA} (see also \cite{ximenaStudia}): for $1 \leq p < \infty$, an operator $u \in {\cal L}(E;F^*)$ is   quasi-$\tau(p)$-summing, in symbols $u \in \Pi_{q\tau(p)}(E;F^*)$, if there exists a constant $C \geq 0$ such that
\begin{equation}\label{eqeu}
\left( \sum\limits_{j=1}^{m} |u(x_j)|^{p}\right)^{\frac{1}{p}} \leq C \cdot \sup\limits_{x^* \in B_{E^*}, y^* \in B_{F^*}} \left(\sum\limits_{j=1}^{m} x^*(x_j)y^*(y_{j})|^{p} \right)^{\frac{1}{p}},
\end{equation}
for all $m \in \mathbb{N}$, $x_j \in E$, $y_{j} \in F$,  $j = 1,\ldots,m$. The infimum of such constants $C$ is denoted by $\|u\|_{q\tau(p)}$. A routine computation shows that $(\Pi_{q\tau(p)}(E;F^*), \|\cdot\|_{q\tau(p)})$ is a Banach space. Given $v \in \mathcal{L}(E;F)$ and $u \in \Pi_{q\tau(p)}(F;G^*)$, let $C$ be as in (\ref{eqeu}). For  $m \in \mathbb{N}$, $x_j \in E$, $w_{j} \in G$,  $j = 1,\ldots,m$,
\begin{align*}
\left(\sum\limits_{j=1}^{m} |(T \circ v)(x_{j})(w_{j})|^{p}\right)^{\frac{1}{p}}
 &\leq C \cdot \sup\limits_{y^* \in B_{F^*},w^* \in B_{G^*}} \left(\sum\limits_{j=1}^{m} |y^*(v(x_{j}))w^*(w_{j})|^{p}\right)^{\frac{1}{p}}  \\
& = C \cdot \|v\| \cdot \sup\limits_{y^* \in B_{F^*},w^* \in B_{G^*}} \left(\sum\limits_{j=1}^{m} \left|\left(y^* \circ \frac{v}{\|v\|}\right)(x_{j})w^*(w_{j})\right|^{p}\right)^{\frac{1}{p}} \\
& \leq C \cdot \|v\| \cdot \sup\limits_{x^* \in B_{E^*},w^* \in B_{G^*}} \left(\sum\limits_{j=1}^{m} |x^*(x_{j})w^*(w_{j})|^{p}\right)^{\frac{1}{p}},
\end{align*}
which shows that $(u \circ v) \in \Pi_{q\tau(p)}(E;G^*)$ and $\|u \circ v \|_{q\tau(p)} \leq \|u\|_{q\tau(p)} \cdot \|v\|.$\\
(c) The class of weak$^*$-sequentially compact operators $u \colon E \longrightarrow F^*$  ($(u(x_j))_{j=1}^\infty$ admits a weak$^*$ convergent subsequence in $F^*$ for every bounded sequence $(x_j)_{j=1}^\infty$ in $E$) satisfies the left ideal property obviously. In the proof of \cite[Proposition 3.3.2]{ximenaStudia} it is proved that it does not satisfy the right ideal property.
\end{example}

For the problem of extending a left semi-operator ideal to a Banach operator ideal we refer to \cite{ximenaStudia}.

Using standard techniques from the theory of Banach operator ideals it is easy to prove the following lemma. This is where inequality (\ref{semi}) comes into play.

\begin{lemma}\label{ammel} Let $\alpha$ be a left semi-operator ideal and let $v\colon E \longrightarrow F$ be an (isometric) isomorphism. Then, for every Banach space $G$, the composition operator
$$J_{v}\colon \mathcal{L}_{\alpha}(F;G^*) \longrightarrow \mathcal{L}_{\alpha}(E;G^*)~,~J_v(T)= T \circ v,$$
is an (isometric) isomorphism as well.
\end{lemma}

Suppose that the linear Borel transform
\begin{equation}\label{borlin}{\cal B} \colon ({\cal I}(E,F),\|\cdot\|_{\cal I})^* \longrightarrow {\cal L}_\alpha(E^*;F^*)~,~{\cal B}(\phi)(x^*)(y) = \phi(x^*\otimes y),
\end{equation}
is an isometric isomorphism for some Banach spaces $E$ and $F$, some normed class ${\cal I}(E,F)$ of operators from $E$ to $F$ and some normed class ${\cal L}_\alpha(E^*;F^*)$ of operators from $E^*$ to $F^*$. Next theorem shows how this linear representation can be used to represent linear functionals on composition ideals of multilinear operators, which go back to Pietsch \cite{pietsch83}: Given a Banach operator ideal $({\cal I},\|\cdot\|_{\cal I})$, a multilinear operator $A \in {\cal L}(E_1, \ldots, E_n;F)$ belongs to ${\cal I}\circ {\cal L} $ if there are a Banach space $G$, and operators $B \in {\cal L}(E_1, \ldots, E_n;G)$ and $u \in {\cal I}(G;F)$ such that $A = u \circ B$. In this case,
$$\|A\|_{{\cal I}\circ {\cal L}} := \inf\{\|B\|\cdot\|u\|_{\cal I} : A = u \circ B, u \in {\cal I}\}.  $$
It is well known that ${\cal I}\circ {\cal L}$ is a Banach hyper-ideal of multilinear operators \cite{prims, pietsch83}, hence ${\cal I}\circ {\cal L}(E_1, \ldots, E_n;F)$ satisfies the conditions of Proposition \ref{firstpro} for all Banach spaces $E_1, \ldots, E_n,F$.

By a geometric property of Banach spaces we mean a property that is invariant under isometric isomorphisms.

\begin{theorem}\label{tthh} Let $(\mathcal{I}, \|\cdot\|_{\cal I})$ be a Banach operator ideal,  $\alpha$ be a left semi-operator ideal and $P_{1}$ and $P_{2}$ be geometric properties of Banach spaces. The following are equivalent:\\
{\rm (a)} For all Banach spaces $E$ and $F$ such that $E^*$ has $P_{1}$ and $F$ has $P_{2}$, the linear Borel transform {\rm (\ref{borlin})}
%$${\cal B} \colon ({\cal I}(E,F),\|\cdot\|_{\cal I})^* \longrightarrow {\cal L}_\alpha(E^*;F^*)~,~{\cal B}(\phi)(x^*)(y) = \phi(x^*\otimes y),  $$
is an isometric isomorphism.\\
{\rm (b)} For every $n \in \mathbb{N}$ and all Banach spaces  $E_{1},\ldots,E_{n}$ and $F$ such that $\mathcal{L}(E_{1},\ldots,E_{n})$ has $P_{1}$ and $F$ has $P_{2}$, the hyper-Borel transform
\begin{equation}\label{btra}{\cal B}_n \colon (\mathcal{I \circ L}(E_{1},\ldots,E_{n};F),\|\cdot\|_{\mathcal{I \circ L}})^* \longrightarrow \mathcal{L}_{\alpha}(\mathcal{L}(E_{1},\ldots,E_{n});F^*)~,~ {\cal B}_n(\phi)(A)(y) = \phi(A \otimes y),
\end{equation}
is an isometric isomorphism.\\
{\rm (c)} For some $n \geq 2$ and all Banach spaces  $E_{1},\ldots,E_{n}$ and $F$ such that $\mathcal{L}(E_{1},\ldots,E_{n})$ has $P_{1}$ and $F$ has $P_{2}$, the hyper-Borel transform {\rm (\ref{btra})} is an isometric isomorphism.
\end{theorem}

\begin{proof} In this proof the components of $\cal I$ are endowed with the norm $\|\cdot\|_{\cal I}$ and the components of ${\cal I}\circ {\cal L}$ with $\|\cdot\|_{{\cal I}\circ {\cal L}}$.\\
(a) $\Longrightarrow$ (b) Let $n \in \mathbb{N}$ and suppose that $\mathcal{L}(E_{1},\ldots,E_{n})$ has $P_{1}$ and $F$ has $P_{2}$. By \cite[Propositions 3.2 and 3.7]{prims} we know that the linearization operator
$$L\colon   \mathcal{I \circ L}(E_{1},\ldots,E_{n};F) \longrightarrow \mathcal{I}(E_{1} \hat{\otimes}_{\pi} \cdots \hat{\otimes}_{\pi} E_{n} ;F)~,~ L(A) = A_L,$$
is an isometric isomorphism. Hence, the adjoint of its inverse
$$(L^{-1})^*\colon(\mathcal{I \circ L}(E_{1},\ldots,E_{n};F))^* \longrightarrow \mathcal{I}(E_{1} \hat{\otimes}_{\pi} \cdots \hat{\otimes}_{\pi} E_{n} ;F)^*$$
is also an isometric isomorphism. Since the linearization operator $\ell \colon \mathcal{L}(E_{1},\ldots,E_{n}) \longrightarrow (E_{1} \hat{\otimes}_{\pi} \cdots \hat{\otimes}_{\pi} E_{n})^*$ is an isometric isomorphism \cite{ryan}, we have that: (i) $(E_{1} \hat{\otimes}_{\pi} \cdots \hat{\otimes}_{\pi} E_{n})^*$ has $P_1$, so the linear Borel transform ${\cal B} \colon \mathcal{I}(E_{1} \hat{\otimes}_{\pi} \cdots \hat{\otimes}_{\pi} E_{n};F)^*
 \longrightarrow \mathcal{L}_{\alpha}((E_{1} \hat{\otimes}_{\pi} \cdots \hat{\otimes}_{\pi} E_{n})^*;F^*)$ is an isometric isomorphism by asumption; (ii) the composition operator
$$ J_{\ell} \colon \mathcal{L}_{\alpha}((E_{1} \hat{\otimes}_{\pi} \cdots \hat{\otimes}_{\pi} E_{n})^*;F^*) \longrightarrow \mathcal{L}_{\alpha}(\mathcal{L}(E_{1},\ldots,E_{n});F^*))~,~
 J_{\ell}(u) = u \circ \ell,
 $$
is an isometric isomorphism by Lemma \ref{ammel}. It is enough to prove that the diagram %Bearing in mind the following chain of isometric isomorphisms
\begin{equation*}
\begin{gathered}
\xymatrix@C2pt@R15pt{
           \mathcal{I}(E_{1} \hat{\otimes}_{\pi} \cdots \hat{\otimes}_{\pi} E_{n};F)^* \ar[rrrrrrrrrr]^*{\cal B} & & & & &  & & & & & \mathcal{L}_{\alpha}((E_{1} \hat{\otimes}_{\pi} \cdots \hat{\otimes}_{\pi} E_{n})^*;F^*) \ar@/^/[dd]^*{J_\ell}&\\
&  &  &  &  & & & & &  & & & & &  &\\
         (\mathcal{I \circ L}(E_{1},\ldots,E_{n};F))^*\ar@/_/[uu]^*{(L^{-1})^*}\ar[rrrrrrrrrr]^*{{\cal B}_n}  & & & & & & & & & & \mathcal{L}_{\alpha}(\mathcal{L}(E_{1},\ldots,E_{n});F^*)
}
\end{gathered}
\end{equation*}
is commutative:
%$$(\mathcal{I \circ L})(E_{1},\ldots,E_{n};F)^* \stackrel{(L^{-1})^*}{\longrightarrow} \mathcal{I}(E_{1} \hat{\otimes}_{\pi} \cdots \hat{\otimes}_{\pi} E_{n};F)^*
% \stackrel{{\cal B}}{\longrightarrow} \mathcal{L}_{\alpha}((E_{1} \hat{\otimes}_{\pi} \cdots \hat{\otimes}_{\pi} E_{n})^*;F^*) $$ $$\stackrel{J_\ell}{\longrightarrow} \mathcal{L}_{\alpha}(\mathcal{L}(E_{1},\ldots,E_{n});F^*),$$
%all that is left to do is to prove that $J_\ell \circ {\cal B} \circ (L^{-1})^* = {\cal B}_n$:
for $\phi \in (\mathcal{I} \circ \mathcal{L})(E_{1},\ldots,E_{n};F)^*$, $A \in \mathcal{L}(E_{1},\ldots,E_{n})$ and $y \in F$,
\begin{align*}
\left(J_{\ell} \circ {\cal B} \circ  (L^{-1})^*\right)&(\phi)(A)(y) = \left((J_{\ell}({\cal B}((L^{-1})^*(\phi))))(A)\right)(y) = \left(({\cal B}((L^{-1})^*(\phi)))(A_{L})\right)(y) \\& = (L^{-1})^*(\phi)(A_{L} \otimes y) = \phi((L^{-1})(A_{L} \otimes y)) = \phi(A \otimes y) = {\cal B}_n(\phi)(A)(y),
\end{align*}
because $L(A \otimes y) = (A \otimes y)_L = A_L \otimes y$.\\
(b) $\Longrightarrow$ (c) is obvious. Let us prove (c) $\Longrightarrow$ (a). Suppose that $E^*$ has $P_{1}$ and $F$ has $P_{2}$. A routine computation shows that the operator
$$v \colon E^* \longrightarrow {\cal L}(E, \mathbb{K}, \stackrel{(n-1)}\ldots,{\mathbb{K}})~,~v(x^*)(x, \lambda_1, \ldots, \lambda_{n-1}) =  \lambda_1 \cdots \lambda_{n-1}x^*(x), $$
is an isometric isomorphism, therefore ${\cal L}(E, \mathbb{K}, \stackrel{(n-1)}\ldots,{\mathbb{K}})$ has $P_1$ and the hyper Borel transform ${\cal B}_n \colon (\mathcal{I \circ L}(E, \mathbb{K}, \stackrel{(n-1)}\ldots,{\mathbb{K}};F))^* \longrightarrow \mathcal{L}_{\alpha}({\cal L}(E, \mathbb{K}, \stackrel{(n-1)}\ldots,{\mathbb{K}});F^*)$ is an isometric isomorphism by assumption. Furthermore, the composition operator
$$J_{v}\colon \mathcal{L}_{\alpha}({\cal L}(E, \mathbb{K}, \stackrel{(n-1)}\ldots,{\mathbb{K}});F^*) \longrightarrow \mathcal{L}_{\alpha}(E^*;F^*)~,~u \mapsto u \circ v, $$
 is an isometric isomorphism by Lemma \ref{ammel}. The linearization operator $$\ell \colon {\cal L}(E, \mathbb{K}, \stackrel{(n-1)}\ldots,{\mathbb{K}}) \longrightarrow (E \hat{\otimes}_{\pi} \mathbb{K} \hat{\otimes}_{\pi} \stackrel{(n-1)}{\cdots} \hat{\otimes}_{\pi} \mathbb{K})^*$$ is an isometric isomorphism, so is the composition $\ell \circ v \colon E^* \longrightarrow (E \hat{\otimes}_{\pi} \mathbb{K} \hat{\otimes}_{\pi} \stackrel{(n-1)}{\cdots} \hat{\otimes}_{\pi} \mathbb{K})^*$. Let $B_L \colon E \hat{\otimes}_{\pi} \mathbb{K} \hat{\otimes}_{\pi} \stackrel{(n-1)}{\cdots} \hat{\otimes}_{\pi} \mathbb{K} \longrightarrow E$ be the linearization of the continuous multilinear operator
 $$B \colon E \times \mathbb{K}^{n-1}\longrightarrow E~,~B(x, \lambda_1, \ldots, \lambda_{n-1}) = \lambda_1 \cdots \lambda_{n-1}x. $$
 For $x^* \in E^*$ and $x \otimes \lambda_1 \otimes \cdots \otimes \lambda_{n-1} \in E \hat{\otimes}_{\pi} \mathbb{K} \hat{\otimes}_{\pi} \stackrel{(n-1)}{\cdots} \hat{\otimes}_{\pi} \mathbb{K}$,
\begin{align*}B_L^*(x^*)(x \otimes \lambda_1 \otimes \cdots \otimes \lambda_{n-1})& = x^*(B(x, \lambda_1, \ldots, \lambda_{n-1}))  =\lambda_1 \cdots \lambda_{n-1} x^*(x) \\&= v(x^*)(x, \lambda_1, \ldots, \lambda_{n-1})= v(x^*)_L(x \otimes \lambda_1 \otimes \cdots \otimes\lambda_{n-1}) \\&=(\ell \circ v)(x^*)(x \otimes \lambda_1 \otimes \cdots \otimes \lambda_{n-1}).
\end{align*}
Since both $B_L^*$ and $\ell \circ v$ are bounded linear operators and the finite sums of elementary tensors are dense in $E \hat{\otimes}_{\pi} \mathbb{K} \hat{\otimes}_{\pi} \stackrel{(n-1)}{\cdots} \hat{\otimes}_{\pi} \mathbb{K}$, it follows that
$B_L^* = \ell \circ v$, from which we conclude that $B_L$ is an isometric isomorphism. Since ${\cal I}$ is a Banach operator ideal, the composition operator
$$J_{(B_L)^{-1}} \colon {\cal I}(E \hat{\otimes}_{\pi} \mathbb{K} \hat{\otimes}_{\pi} \stackrel{(n-1)}{\cdots} \hat{\otimes}_{\pi} \mathbb{K};F) \longrightarrow  {\cal I}(E;F) ~,~u \mapsto u \circ (B_L)^{-1},$$
is an isometric isomorphism too. Calling on \cite[Propositions 3.2 and 3.7]{prims} once again, the linearization operator  $L\colon   \mathcal{I \circ L}(E, \mathbb{K}, \stackrel{(n-1)}\ldots,{\mathbb{K}};F) \longrightarrow \mathcal{I}(E \hat{\otimes}_{\pi} \mathbb{K} \hat{\otimes}_{\pi} \stackrel{(n-1)}{\cdots} \hat{\otimes}_{\pi} \mathbb{K} ;F)$ is an isometric isomorphism, so are $J_{(B_L)^{-1}}\circ L \colon \mathcal{I \circ L}(E, \mathbb{K}, \stackrel{(n-1)}\ldots,{\mathbb{K}};F) \longrightarrow {\cal I}(E;F)$ and its adjoint $(J_{(B_L)^{-1}}\circ L)^* \colon {\cal I}(E;F)^* \longrightarrow (\mathcal{I \circ L}(E, \mathbb{K}, \stackrel{(n-1)}\ldots,{\mathbb{K}};F))^*$. We just have to check that the diagram
\begin{equation*}
\begin{gathered}
\xymatrix@C2pt@R15pt{
           (\mathcal{I \circ L}(E, \mathbb{K}, \stackrel{(n-1)}\ldots,{\mathbb{K}};F))^* \ar[rrrrrrrrrr]^*{{\cal B}_n} & & & & &  & & & & & \mathcal{L}_{\alpha}({\cal L}(E, \mathbb{K}, \stackrel{(n-1)}\ldots,{\mathbb{K}});F^*) \ar@/^/[dd]^*{J_{v}}&\\
&  &  &  &  & & & & &  & & & & &  &\\
        {\cal I}(E;F)^*\ar@/_/[uu]^*{(J_{(B_L)^{-1}}\circ L)^*}\ar[rrrrrrrrrr]^*{\cal B}  & & & & & & & & & & \mathcal{L}_{\alpha}(E^*;F^*)
}
\end{gathered}
\end{equation*}
is commutative. First note that given $x^* \in E^*$ and $y \in F$,
\begin{align*}((v(x^*)_L\otimes y)\circ (B_L)^{-1})(x)&= (v(x^*)_L\otimes y)((B_L)^{-1}(x))  = v(x^*)_L((B_L)^{-1}(x))y\\
&= v(x^*)_L(x\otimes1 \otimes \cdots \otimes 1)y = v(x^*)(x,1, \ldots, 1)y = x^*(x)y
\end{align*}
for every $x \in E$
because $B_L(x\otimes1 \otimes \cdots \otimes 1) = B(x,1,\ldots, 1) = x$. This proves that $(v(x^*)_L\otimes y)\circ (B_L)^{-1} = x^* \otimes y$. So, for $\phi \in {\cal I}(E;F)^*$, $x^* \in E^*$ and $y \in F$,
\begin{align*} (J_{v} \circ {\cal B}_n & \circ (J_{(B_L)^{-1}}\circ L)^*)(\phi)(x^*)(y) = (J_{v}({\cal B}_n ((J_{(B_L)^{-1}}\circ L)^*(\phi))))(x^*)(y)\\
& =  (({\cal B}_n ((J_{(B_L)^{-1}}\circ L)^*(\phi)))\circ v)(x^*)(y) = {\cal B}_n ((J_{(B_L)^{-1}}\circ L)^*(\phi))(v(x^*))(y)\\
& = (J_{(B_L)^{-1}}\circ L)^*(\phi)(v(x^*)\otimes y) = \phi((J_{(B_L)^{-1}}\circ L)(v(x^*)\otimes y))\\
&= \phi(J_{(B_L)^{-1}}(L(v(x^*)\otimes y))) = \phi(J_{(B_L)^{-1}}((v(x^*)\otimes y)_L))\\
& = \phi(J_{(B_L)^{-1}}(v(x^*)_L\otimes y)) =  \phi((v(x^*)_L\otimes y)\circ (B_L)^{-1})= \phi(x^* \otimes y) = {\cal B}(\phi)(x^*)(y).
\end{align*}
\end{proof}

\section{Applications}
In this section we give some concrete applications of Theorem \ref{tthh}.

 To simplify the notation and the terminology, the symbol $ [{\cal I}(E,F)]^* \stackrel{{\cal B}}{=} {\cal L}_\alpha(E^*;F^*)$ means that linear Borel transform ${\cal B} \colon ({\cal I}(E,F);\|\cdot\|_{\cal I})^* \longrightarrow {\cal L}_\alpha(E^*;F^*)$ is an isometric isomorphism; and the symbol $[{\cal H}(E_{1},\ldots,E_{n};F)]^* \stackrel{{\cal B}_n}{=} \mathcal{L}_{\alpha}(\mathcal{L}(E_{1},\ldots,E_{n});F^*)$  means that the hyper-Borel transform ${\cal B}_n \colon({\cal H}(E_{1},\ldots,E_{n};F), \|\cdot\|_{\cal H})^* \longrightarrow \mathcal{L}_{\alpha}(\mathcal{L}(E_{1},\ldots,E_{n});F^*)$ is an isometric isomorphism.

\subsection{Approximable and compact operators}
In this subsection, and only here, the spaces of multilinear operators are considered with the usual operator norm.

We shall use the following symbols:\\
$\bullet$ $\overline{\cal F}$ = closed ideal of linear operators that can be approximated in the usual norm by finite rank operators; \\
$\bullet$ $\cal K$ = closed  ideal of compact linear operators;\\
$\bullet$ $({\cal J}, \|\cdot\|_{\cal J})$ = Banach ideal of integral linear operators;\\
$\bullet$ $\overline{{\cal L}_{\cal F}}$ = closed hyper-ideal of multilinear operators that can be approximated in the usual norm by finite rank operators;\\
$\bullet$ ${\cal L}_{\cal K}$ = closed hyper-ideal of compact multilinear operators.

%The reader should bear in mind that the symbol ${\cal L}_{\cal A}$ is sometimes used to denote the (smaller) class of multilinear operators that can be approximated by operators of finite type.

Our first application represents linear functionals on spaces of approximable multilinear operators as integral linear operators.

\begin{theorem}\label{thapp} $[{\overline{{\cal L}_{\cal F}}}(E_{1},\ldots,E_{n};F)]^* \stackrel{{\cal B}_n}{=} {\cal J}(\mathcal{L}(E_{1},\ldots,E_{n});F^*)$ for all Banach spaces $E_1, \ldots,$ $E_n$ and $ F$.
\end{theorem}

\begin{proof} The symbol $E^*\widehat{\otimes}_\varepsilon F$ stands for the completed injective tensor product of $E^*$ and $F$. By \cite[4.2(1)]{df} we know that the operator
 $$T \colon E^*\widehat{\otimes}_\varepsilon F\longrightarrow \overline{\cal F}(E;F)~,~T(x^*\otimes y)(x) = x^*(x)y, $$
is an isometric isomorphism, so is its adjoint
$$T^* \colon \overline{\cal F}(E;F)^* \longrightarrow  (E^*\widehat{\otimes}_\varepsilon F)^*. $$
And \cite[Proposition 10.1]{df} gives that the operator
$$U \colon (E^*\widehat{\otimes}_\varepsilon F)^* \longrightarrow {\cal J}(E^*,F^*)~,~ U(\psi)(x^*)(y) = \psi(x^* \otimes y), $$
%$$V \colon {\cal L}(E^*,F;\mathbb{K}) \longrightarrow {\cal L}(E^*;F^*)~,~V(A)(\varphi)(y) = A(\varphi,y), $$
is also an isometric isomorphism. % (actually $U$ is the inverse of the linearization operator $A \mapsto A_L)$.
Let us see that ${\cal B} = U\circ T^*$: for $\phi \in \overline{\cal F}(E,F)^*$, $x^*\in E^*$ and $y \in F$,
$$( U\circ T^*)(\phi)(x^*)(y) = U(T^*(\phi))(x^*)(y) = T^*(\phi)(x^*\otimes y)= \phi(x^*(x)y)= {\cal B}(\phi)(x^*)(y).  $$
This shows that $[\overline{\cal F}(E;F)]^* \stackrel{{\cal B}}{=} {\cal J}(E^*,F^*)$ for all Banach spaces $E$ and $F$. Since $\cal J$ is a Banach operator ideal, by Theorem \ref{tthh} we have  $[\overline{\cal F}\circ {\cal L}(E_{1},\ldots,E_{n};F)]^* \stackrel{{\cal B}_n}{=} \mathcal{J}(\mathcal{L}(E_{1},\ldots,E_{n});F^*)$ for all $E_1, \ldots, E_n, F$. The proof is complete because $\overline{{\cal L}_{\cal F}} = \overline{\cal F}\circ {\cal L}$ isometrically. The polynomial case of this equality can be found in \cite[Theorem 2.2]{leticia}, the multilinear case is analogous.
\end{proof}

Denote by $\overline{{\cal L}_f}$ the closed ideal of multilinear operators that can be approximated in the usual norm by finite type operators and by ${\cal L}_{\cal J}$ the Banach ideal of integral multilinear operators \cite{ryan}. In \cite{alencar} Alencar proves that $\overline{{\cal L}_f}(E_1, \ldots, E_n;F)^* = {\cal L}_{\cal J}(E_1^*, \ldots, E_n^*;F^*)$ if $E_1, \ldots, E_n, F$ are refexive. Besided of representing linear functionals as linear operators rather than as multilinear operators, our representation in Theorem \ref{thapp} works for all Banach spaces.

In the presence of the approximation property we can describe the linear functionals on spaces of compact multilinear operators as integral linear operators. We are not aware of any other representation theorem for linear functionals on spaces of compact multilinear operators.

\begin{theorem}  If ${\cal L}(E_1, \ldots, E_n)$ or $F^*$ has the approximation property, then $$[{\cal L}_{\cal K}(E_{1},\ldots,E_{n};F)]^* \stackrel{{\cal B}_n}{=} {\cal J}(\mathcal{L}(E_{1},\ldots,E_{n});F^*).$$ %for all Banach spaces $E_1, \ldots E_n, F$.
\end{theorem}

\begin{proof} If $E^*$ of $F$ has the approximation property, then ${\cal K}(E;F) = \overline{\cal F}(E;F)$ \cite[1.e.4, 1.e.5]{lindenstrauss}. The first part of the proof above shows that $[{\cal K}(E;F)]^* \stackrel{{\cal B}}{=} {\cal J}(E^*,F^*)$ whenever $E^*$ of $F$ has the approximation property. Since $\cal J$ is a Banach operator ideal, by Theorem \ref{tthh} we have  $[{\cal K}\circ {\cal L}(E_{1},\ldots,E_{n};F)]^* \stackrel{{\cal B}_n}{=} \mathcal{J}(\mathcal{L}(E_{1},\ldots,E_{n});F^*)$ if ${\cal L}(E_1, \ldots, E_n)$ or $F^*$ has the approximation property. The isometric equality ${\cal L}_{\cal K} = {\cal K}\circ {\cal L}$ is well known (see \cite[Lemma 4.1]{teseryan} or \cite[Proposition 3.4(a)]{mujicatams}), so the proof is complete.
\end{proof}

\subsection{Hyper-nuclear and hyper-$(s,r)$-nuclear operators}

As announced in Section 2, we shall prove in this section that $[{\cal HN}(E_1, \ldots, E_n;F)^*] \stackrel{{\cal B}_n}{=}\mathcal{L}(\mathcal{L}(E_{1},\ldots,E_{n});F^*)$. To do so using Theorem \ref{tthh}, we have to check that ${\cal L}_{\cal HN}$ is a composition hyper-ideal. We shall prove it in a much more general setting.

We use the convention $1/\infty = 0$. Let $s \in [1,\infty)$ and $r \in [1,\infty]$ be such that $1 \leq \frac{1}{s} + \frac{1}{r}$. According to \cite{ewertonlaa}, an $n$-linear operator $A \in \mathcal{L}(E_{1},\ldots,E_{n};F)$ is  {\it hyper-$(s,r)$-nuclear} if there are sequences $(\lambda_{j})_{j=1}^{\infty} \in \ell_{s}$, $(T_{j})_{j=1}^{\infty} \in \ell_{r}^{w}(\mathcal{L}(E_{1},\ldots,E_{n}))$ and $(y_{j})_{j=1}^{\infty} \in \ell_{\infty}(F)$ such that
\begin{equation}\label{rep hiper-nuclear}
A(x_{1},\ldots,x_{n}) = \sum\limits_{j=1}^{\infty} \lambda_{j} T_{j} \otimes y_{j} (x_{1},\ldots,x_{n}),
\end{equation}
for every $(x_{1},\ldots,x_{n}) \in E_{1} \times \cdots \times E_{n}.$ The space of all these operators is denoted by  ${\mathcal{HN}_{s,r}}(E_{1},\ldots,E_{n};F)$ and its norm is defined by
\begin{center}
$\|A\|_{\mathcal{HN}_{s,r}} = \inf\{ \|(\lambda_{j})_{j=1}^{\infty} \|_{s} \cdot \|(T_{j})_{j=1}^{\infty} \|_{w,r} \cdot \|(y_{j})_{j=1}^{\infty} \|_{\infty} \},$
\end{center}
where the infimum is taken over all hyper-$(s,r)$-nuclear representations of $A$ as in  (\ref{rep hiper-nuclear}).

Taking $s=1$ and $r = \infty$ and recalling that $\ell_\infty^w(E) = \ell_\infty(E)$, we recover the class of hyper-nuclear operators considered in Section 2, that is, ${\cal HN} = \mathcal{HN}_{1,\infty}$.

In \cite{ewertonlaa} it is proved that $\mathcal{HN}_{s,r}$ is a Banach hyper-ideal of multilinear operators. In the linear case $n = 1$ we denote the Banach operator ideal of $(s,r)$-nuclear linear operators simply by $\mathcal{N}_{s,r}$. In the same fashion, the Banach ideal of nuclear operators $\mathcal{N}_{1,\infty}$ is denoted by $\cal N$.

\begin{lemma}\label{lemaHN} $\mathcal{HN}_{s,r} = \mathcal{N}_{s,r} \circ {\cal L}$ isometrically and, in particular, $\mathcal{HN} = \mathcal{N} \circ {\cal L}$ isometrically.
\end{lemma}

\begin{proof} Given Banach spaces $E_1, \ldots, E_n, F$, an operator $A \in \mathcal{N}_{s,r} \circ \mathcal{L}(E_{1},\ldots,E_{n};F)$ and $\varepsilon > 0$, there are a Banach space $G$  and operators $B \in \mathcal{L}(E_{1},\ldots,E_{n};G)$ and $u \in \mathcal{N}_{s,r}(G;F)$ such that $A = u \circ B$. We can take sequences $(\lambda_{j})_{j=1}^{\infty} \in \ell_{s}$, $(y_{j})_{j=1}^{\infty} \in \ell_{\infty}(F)$ and $(w^*_{j})_{j=1}^{\infty} \in \ell_{r}^{w}(G^*)$ such that $u = \sum\limits_{j=1}^{\infty} \lambda_{j} w^*_{j} \otimes y_{j}$ and
\begin{equation}
\sum\limits_{j=1}^{\infty} \|(\lambda_j)_{j=1}^\infty\|_{s} \cdot \|(w^*_{j})_{j=1}^\infty\|_{w,r} \cdot \|(y_{j})_{j=1}^\infty\|_{\infty} \leq (1 + \varepsilon)\|u\|_{\mathcal{N}_{s,r}}.
\end{equation}
For any $(x_{1},\ldots,x_{n}) \in E_{1}, \times \cdots \times E_{n}$,
$$A(x_{1},\ldots,x_{n}) = \sum\limits_{j=1}^{\infty} \lambda_{j}(w^*_{j} \circ B) \otimes y_{j}(x_{1},\ldots,x_{n}). $$
It is clear that
$$B^* \colon G^* \longrightarrow {\cal L}(E_1, \ldots, E_n)~,~B^*(w^*)(x_1, \ldots, x_n) = w^*(B(x_1, \ldots, x_n)), $$
is a bounded linear operator and $\|B^*\| = \|B\|$. Therefore, $(w^*_{j} \circ B)_{j=1}^{\infty} = (B^*(w^*_j))_{j=1}^\infty \in \ell_{r}^{w}(\mathcal{L}(E_{1},\ldots,E_{n}))$ and $\|(w^*_{j} \circ B)_{j=1}^{\infty}\|_{w,r} \leq \|B\| \cdot \|(w^*_{j})_{j=1}^{\infty}\|_{w,r}$ (see \cite[p.\,92]{df}). It follows that $A = \sum\limits_{j=1}^{\infty} \lambda_{j}(w^*_{j} \circ B) \otimes y_{j}$ is a hyper-$(s,r)$-nuclear representation for $A$ and
\begin{align*}
\|A\|_{\mathcal{HN}_{s,r}} \leq  \sum\limits_{j=1}^{\infty} \|(\lambda_{j})_{j=1}^{\infty}\|_{s} \cdot \|B\| \cdot \|(w^*_{j})_{j=1}^{\infty}\|_{w,r} \cdot \|(y_{j})_{j=1}^{\infty}\|_{\infty}
 \leq \|B\| \cdot (1 + \varepsilon) \|u\|_{\mathcal{N}_{s,r}}.
\end{align*}
We have $A \in {\mathcal{HN}_{s,r}}(E_{1},\ldots,E_{n};F)$ and, letting $\varepsilon \longrightarrow 0$, $\|A\|_{\mathcal{HN}_{s,r}} \leq \|B\| \cdot \|u\|_{\mathcal{N}_{s,r}}$. Taking now the infimum over all such factorizations $A = u \circ B$ we conclude that $\|A\|_{\mathcal{HN}_{s,r}} \leq \|A\|_{\mathcal{N}_{s,r} \circ \mathcal{L}}$.

Conversely, given $A \in \mathcal{HN}_{s,r}(E_{1},\ldots,E_{n};F)$ and $\varepsilon > 0$, there are sequences $(\lambda_{j})_{j=1}^{\infty} \in \ell_{s}$, $(B_{j})_{j=1}^{\infty} \in \ell_{r}^{w}(\mathcal{L}(E_{1},\ldots,E_{n}))$ and $(y_{j})_{j=1}^{\infty} \in \ell_{\infty}(F)$ such that $A = \sum\limits_{j=1}^{\infty} \lambda_{j} B_{j} \otimes y_{j}$ and
$$\sum\limits_{j=1}^{\infty} \|(\lambda_{j})_{j=1}^{\infty}\|_{s} \cdot \|(B_{j})_{j=1}^{\infty}\|_{w,r} \cdot \|(y_{j})_{j=1}^{\infty}\|_{\infty} \leq (1 + \varepsilon)\|A\|_{\mathcal{HN}_{s,r}}.$$
From
\begin{align*}
\lim_m\left\| A - \sum\limits_{j=1}^{m} \lambda_{j}B_{j} \otimes y_{j} \right\|_{\mathcal{HN}_{s,r}}& = \lim_m\left\| \sum\limits_{j=m+1}^{\infty} \lambda_{j}B_{j} \otimes y_{j} \right\|_{\mathcal{HN}_{s,r}} \\
& \leq \lim_m\|(\lambda_{j})_{j=m+1}^{\infty}\|_{s} \cdot \|(B_{j})_{j=m+1}^{\infty}\|_{w,r} \cdot \|(y_{j})_{j=m+1}^{\infty}\|_{\infty} \\
& \leq \|(B_{j})_{j=1}^{\infty}\|_{w,r} \cdot \|(y_{j})_{j=1}^{\infty}\|_{\infty}\cdot\lim_m\|(\lambda_{j})_{j=m+1}^{\infty}\|_{s} =0,
 %\left( \sum\limits_{j=m+1}^{\infty} |\lambda_{j}|^{s} \right)^{\frac{1}{s}} \cdot \sup\limits_{\varphi} \left( \sum\limits_{j=1}^{\infty} |\varphi(B_{j})|^{r}\right)^{\frac{1}{r}} \cdot \sup\limits_{j} \|y_{j}\|
\end{align*}
we have that $A = \sum\limits_{j=1}^{\infty} \lambda_{j} B_{j} \otimes y_{j}$ in the norm $\|\cdot\|_{\mathcal{HN}_{s,r}}$. Since $\|\cdot\| \leq \|\cdot\|_{\mathcal{HN}_{s,r}}$ because $\mathcal{HN}_{s,r}$ is a Banach hyper-ideal, the convergence also occurs in the usual operator norm on $\mathcal{L}(E_{1},\ldots,E_{n})$. Since the correspondence
$$B \in \mathcal{L}(E_{1},\ldots,E_{n};F) \mapsto B_{L} \in \mathcal{L}(E_{1} \widehat{\otimes}_{\pi} \cdots \widehat{\otimes}_{\pi} E_{n};F)$$
is an isomorphism, hence linear and continuous, it follows that
\begin{align}\label{yuyu}
A_{L} = \left( \sum_{j=1}^\infty\lambda_{j} B_{j} \otimes y_{j} \right)_{L} =  \sum_{j=1}^\infty(\lambda_{j} B_{j} \otimes y_{j})_L % \lim\limits_{m} \sum\limits_{j=1}^{m} \lambda_{j} (B_{j} \otimes y_{j})_{L}\\
%& = \lim\limits_{m} \sum\limits_{j=1}^{m} \lambda_{j} (B_{j})_{L} \otimes y_{j}
 = \sum\limits_{j=1}^{\infty} \lambda_{j} (B_{j})_{L} \otimes y_{j}
\end{align}
in $\mathcal{L}(E_{1} \widehat{\otimes}_{\pi} \cdots \widehat{\otimes}_{\pi} E_{n};F)$. Applying now the isometric isomorphism $B \in \mathcal{L}(E_{1},\ldots,E_{n}) \mapsto$  $ B_{L} \in (E_{1} \widehat{\otimes}_{\pi} \cdots \widehat{\otimes}_{\pi} E_{n})^*$ and calling on \cite[p.\,92]{df} once again we get  $$ ((B_{j})_{L})_{j=1}^\infty \in \ell_{r}^{w}((E_{1} \widehat{\otimes}_{\pi} \cdots \widehat{\otimes}_{\pi} E_{n})^*){ \rm ~ and~} \|((B_{j})_{L})_{j=1}^\infty\|_{w,r} = \|(B_{j})_{j}\|_{w,r}.$$
This proves that (\ref{yuyu}) is an $(s,r)$-nuclear representation for $A$ and
 \begin{align*}
\|A_{L}\|_{\mathcal{N}_{s,r}}& \leq \sum\limits_{j=1}^{\infty} \|(\lambda_{j})_{j=1}^{\infty}\|_{s} \cdot \|((B_{j})_{L})_{j=1}^{\infty}\|_{w,r} \cdot \|(y_{j})_{j=1}^{\infty}\|_{\infty}\\
& = \sum\limits_{j=1}^{\infty}  \|(\lambda_{j})_{j=1}^{\infty}\|_{s} \cdot \|(B_{j})_{j=1}^{\infty}\|_{w,r} \cdot \|(y_{j})_{j=1}^{\infty}\|_{\infty} \leq (1 + \varepsilon)\|A\|_{\mathcal{HN}_{s,r}}.
\end{align*}
So, $A_{L} \in \mathcal{N}_{s,r}(E_{1} \widehat{\otimes}_{\pi} \cdots \widehat{\otimes}_{\pi} E_{n};F)$ and, making $\varepsilon \longrightarrow 0$, $\|A_{L}\|_{\mathcal{N}_{s,r}} \leq \|A\|_{\mathcal{HN}_{s,r}}$. By \cite[Propositions 3.2 and 3.7]{prims} we conclude that $A \in \mathcal{N}_{s,r} \circ {\cal L}(E_{1},\ldots,E_{n};F)$ and $\|A\|_{\mathcal{N}_{s,r} \circ {\cal L}} \leq \|A\|_{\mathcal{HN}_{s,r}}$.
\end{proof}

%\begin{remark}\rm Sometimes the nuclear norm of a nuclear linear operator (case $n=s=1$ and $r = \infty$) is defined as an infimum of sums, instead of an infimum of products. If one is unease with this in the proof above, see in [DF p.\,291] that the nuclear norm can be see as an infimum of products, or, alternatively, see in [Ewerton laa??] that the hyper-nuclear norm of a hyper-nuclear multilinear operator can be seen as an infimum of sums.
%\end{remark}

The announced representation of linear functionals on spaces of hyper-nuclear multilinear operators as bounded linear operators can be proved now.

\begin{theorem}\label{repHN}  If ${\cal L}(E_1, \ldots, E_n)$ or $F^*$ has the approximation property, then $$[{\cal HN}(E_{1},\ldots,E_{n};F)]^* \stackrel{{\cal B}_n}{=} {\cal L}(\mathcal{L}(E_{1},\ldots,E_{n});F^*).$$
\end{theorem}

\begin{proof}  We believe it is known that $[{\cal N}(E;F)]^* \stackrel{{\cal B}}{=} {\cal L}(E^*,F^*)$ if $E^*$ or $F$ has the approximation property. For the sake of completeness, we give a short reasoning.
By \cite[5.6 Corollary 1, p.\,65]{df}, the operator
$$T \colon E^*\widehat{\otimes}_\pi F\longrightarrow {\cal N}(E;F)~,~T(x^*\otimes y)(x) = x^*(x)y, $$
is an isometric isomorphism, so is its adjoint $T^* \colon {\cal N}(E;F)^* \longrightarrow  (E^*\widehat{\otimes}_\pi F)^*. $
It is well know that the operators
$$U \colon (E^*\widehat{\otimes}_\pi F)^* \longrightarrow {\cal L}(E^*,F;\mathbb{K})~,~ U(\psi)(x^*,y) = \psi(x^* \otimes y), $$
$$V \colon {\cal L}(E^*,F;\mathbb{K}) \longrightarrow {\cal L}(E^*;F^*)~,~V(A)(x^*)(y) = A(x^*,y), $$
are isometric isomorphisms \cite{ryan, mujica} (actually $U$ is the inverse of the linearization operator $A \mapsto A_L)$. It is enough to show that ${\cal B} = V \circ U \circ T^*$. Indeed, for $\phi \in {\cal N}(E,F)^*, x^* \in E^*$ and $y \in F$,
\begin{align*} (V \circ U \circ T^*)(\phi)(x^*)(y) &=V(U(T^*(\phi)))(x^*)(y) = U(T^*(\phi))(x^*, y) \\&= T^*(\phi)(x^*\otimes y) = \phi(T(x^*\otimes y)) = {\cal B}(\phi)(x^*)(y).
\end{align*}
Theorem \ref{tthh} gives $[{\cal N} \circ {\cal L}(E_{1},\ldots,E_{n};F)]^* \stackrel{{\cal B}_n}{=} {\cal L}(\mathcal{L}(E_{1},\ldots,E_{n});F^*)$ if ${\cal L}(E_1, \ldots, E_n)$ or $F^*$ has the approximation property and Lemma \ref{lemaHN} finishes the proof.
\end{proof}

Denote by ${\cal L}_{\cal N}$ the Banach ideal of nuclear multilinear operators introduced by Gupta \cite{gupta2} (see also \cite{alencar, cardim, matos}). Improving previous results of Gupta \cite{gupta2} and Dwyer \cite{dwyer}, Carando and Dimant \cite{cardim} proved that ${\cal L}_{\cal N}(E_1, \ldots, E_n;F)^* = {\cal L}(E_1^*, \ldots, E_n^*;F^*)$ if $E_1, \ldots, E_n$ have the approximation property. Besides of representing linear functionals as linear operators rather than as multilinear operators, our representation in Theorem \ref{repHN} holds for all $E_1, \ldots, E_n$ if $F^*$ has the approximation property.

Now we represent linear functionals on spaces of hyper-$(s,r)$-nuclear multilinear operators as absolutely summing linear operators. For $1 \leq p \leq q < \infty$, by $\Pi_{q,p}$ we denote the Banach ideal of $(q,p)$-summing linear operators \cite[Chapter 10]{djt}, that is, operators that send weakly $p$-summable sequences to absolutely $q$-summable sequences. %Since weakly 1-summable sequences are bounded, we can write $\Pi_{\infty,1}(E;F) = {\cal L}(E;F)$.

%Our first application shows that linear functionals on spaces of hyper-$(s,r)$-nuclear multilinear operators are represented by $(s^*, r)$-summing linear operators, and that linear functionals on the space of hyper-nuclear multilinear operators are represented by bounded linear operators.

\begin{theorem} Let $1 \leq s,r < \infty$ be such that $1 \leq \frac{1}{s} + \frac{1}{r}$ and suppose that ${\cal L}(E_1, \ldots, E_n)$ has the bounded approximation property. Then
$$[{\mathcal{HN}_{s,r}}(E_{1},\ldots,E_{n};F)]^* \stackrel{{\cal B}_n}{=} \Pi_{s^*,r}(\mathcal{L}(E_{1},\ldots,E_{n});F^*) $$
for every Banach space $F$.
%{\rm(b)} The hyper-Borel transform
%is an isometric isomorphism for every Banach space $F$.
\end{theorem}

\begin{proof} The linear case of \cite[Theorem 4.1]{matos} asserts that $[{\cal N}_{s,r}(E;F)]^* \stackrel{{\cal B}}{=} \Pi_{s^*,r}(E^*,F^*)$  whenever $E^*$ has the bounded approximation property. Since $\Pi_{s^*,r}$ is a Banach operator ideal, just combine Theorem \ref{tthh} with Lemma \ref{lemaHN} to obtain the result.
\end{proof}

\subsection{Hyper-$\sigma(p)$-nuclear operators}
The applications given thus far concern well studied classes of multilinear operators. In this section we give an application to a new class.

The Banach ideal ${\cal L}_{\sigma(p)}$ of $\sigma(p)$-nuclear multilinear operators was introduced in \cite{ximenaLAA} as linear and multilinear generalizations of the ideal ${\cal N}_{\sigma}$ of $\sigma$-nuclear linear operators studied by Pietsch \cite{livropietsch}. Its linear component, that is the ideal of $\sigma(p)$-nuclear linear operators, shall be denoted by ${\cal N}_{\sigma(p)}$. In this fashion, ${\cal N}_\sigma = {\cal N}_{\sigma(1)}$. In this section we introduce and represent the linear functionals on a new multilinear generalization of ${\cal N}_{\sigma(p)}$.

\begin{definition}\rm
 Let $1 \leq p < \infty$. We say that an $n$-linear operator $A \in \mathcal{L}(E_{1},\ldots,E_{n};F)$ is
 {\it hiper-$\sigma(p)$-nuclear} if there are sequences $(\lambda_{j})_{j=1}^{\infty} \in \ell_{p^*}$, $(B_{j})_{j=1}^{\infty}$  in $\mathcal{L}(E_{1},\ldots,E_{n})$ and
 $(y_{j})_{j=1}^{\infty}$ in $F$ such that\\
(i) $A(x_{1},\ldots,x_{n}) = \sum\limits_{j=1}^{\infty} \lambda_{j}B_{j}(x_{1},\ldots,x_{n})y_{j}$ for all $x_j \in E_j, j = 1,\ldots, n$;\\
(ii) $\sup\limits_{x_{i} \in B_{E_{i}},y^* \in B_{F^*}} \bigg(\sum\limits_{j=1}^{\infty} |B_{j}(x_{1},\ldots,x_{n})y^*(y_{j})|^{p}\bigg)^{\frac{1}{p}} < \infty$;\\
(iii) $\lim\limits_{m \to \infty} \sup\limits_{x_{i} \in B_{E_{i}},y^* \in B_{F^*}} \bigg(\sum\limits_{j=m}^{\infty} |B_{j}(x_{1},\ldots,x_{n})y^*(y_{j})|^{p}\bigg)^{\frac{1}{p}} = 0$.

In this case we say that $A = \sum\limits_{j=1}^{\infty} \lambda_{j}B_{j} \otimes y_{j}$ is a hyper-$\sigma(p)$-representation of $A$, write $A \in {\cal HN}_{\sigma(p)}(E_1,\ldots, E_n;F)$  and define
$$\|A\|_{\mathcal{HN}\sigma(p)} = \inf \left\{ \|(\lambda_{j})_{j=1}^{\infty}\|_{p^*} \cdot \sup\limits_{x_{i} \in B_{E_{i}},y^* \in B_{F^*}} \left( \sum\limits_{j=1}^{\infty} |B_{j}(x_{1},\ldots,x_{n})y^*(y_{j})|^{p} \right)^{\frac{1}{p}}
\right\},
$$
where the infimum is taken over all hyper-$\sigma(p)$-representation of $A$. H\"older's inequality implies that $\|\cdot\| \leq \|\cdot\|_{\mathcal{HN}\sigma(p)}$.
%\noindent
%onde o ínfimo é tomado sobre todas as representações hiper-$\sigma(p)$-nucleares de $A$.
%Denoting by ${\cal HN}_{\sigma(p)}$
\end{definition}

Soon we will see that $({\cal HN}_{\sigma(p)}, \|\cdot\|_{{\cal HN}_{\sigma(p)}})$ is a Banach hyper-ideal of multilinear operators. Hyper-$\sigma(p)$-nuclear linear operators coincide with the Banach ideal of $\sigma(p)$-nuclear linear operators from \cite{ximenaLAA}. In this case we simply write ${\cal N}_{\sigma(p)}(E;F)$ and $\|\cdot\|_{\sigma(p)}$.

\begin{example}\rm The class ${\cal HN}_{\sigma(p)}$ is neither very small nor very large because it lies between the class ${\cal L}_{\cal F}$ of finite rank operators and its usual norm closure $\overline{{\cal L}_{\cal F}}$. Every $\sigma(p)$-nuclear multilinear operator \cite{ximenaLAA} is hyper-$\sigma(p)$-nuclear but the converse is not true: Let $A$ be any multilinear operator of finite rank that cannot be approximated, in the usual norm, by multilinear operators of finite type. For example:
$$A \colon \ell_2 \times \ell_2 \longrightarrow \mathbb{K}~,~A\left((x_j)_{j=1}^\infty, (y_j)_{j=1}^\infty \right) = \sum_{j=1}^\infty x_jy_j.$$
Then, $A$ is hyper-$\sigma(p)$-nuclear but is not $\sigma(p)$-nuclear because every $\sigma(p)$-nuclear multilinear operator can be approximated by finite type operators in the usual norm.
\end{example}

\begin{lemma}\label{lemmel} If $A = \sum\limits_{j=1}^{\infty} \lambda_{j}A_{j}\otimes y_{j}$ is a hyper-$\sigma(p)$-nuclear representation of $A$, then the convergence occurs in the norm $\|\cdot\|_{{\cal HN}_{\sigma(p)}}$. So, the convergence also occurs in the usual norm and, in particular, ${\cal HN}_{\sigma(p)}$ is a dense subspace of both  $ \overline{{\cal L}_{\cal F}}^{\|\cdot\|_{{\cal HN}_{\sigma(p)}}}$ and $ \overline{{\cal L}_{\cal F}}^{\|\cdot\|}$.
\end{lemma}

\begin{proof} The first assertion follows from
\begin{align*}  \lim_{m \to \infty}&\left\|A - \sum\limits_{j=1}^{m} \lambda_{j} A_{j} \otimes y_{j}\right\|_{\mathcal{HN}\sigma(p)} = \lim_{m \to \infty}\left\|\sum\limits_{j=m+1}^{\infty} \lambda_{j} A_{j} \otimes y_{j}\right\|_{\mathcal{HN}\sigma(p)}\\
& \leq \lim_{m \to \infty}\left[\|(\lambda_{j})_{j=m+1}^{\infty}\|_{p^*} \cdot \sup\limits_{x_{i} \in B_{E_{i}},y^* \in B_{F^*}} \left(\sum\limits_{j=m+1}^{\infty}|A_{j}(x_{1},\ldots,x_{n})y^*(y_{j})|^{p}\right)^{\frac{1}{p}}\right]\\
& \leq \sup\limits_{x_{i} \in B_{E_{i}},y^* \in B_{F^*}} \left(\sum\limits_{j=1}^{\infty}|A_{j}(x_{1},\ldots,x_{n})y^*(y_{j})|^{p}\right)^{\frac{1}{p}} \cdot \lim_{m \to \infty}\|(\lambda_{j})_{j=m+1}^{\infty}\|_{p^*} = 0
\end{align*}
because $(\lambda_{j})_{j=1}^{\infty} \in \ell_{p^*}$. The second assertion now follows from the inequality $\|\cdot\| \leq \|\cdot\|_{\mathcal{HN}\sigma(p)}$.
%$\bigg|\bigg|A - \sum\limits_{j=1}^{n} \lambda_{j} A_{j} \otimes y_{j}\bigg|\bigg|_{\mathcal{H}\sigma(p)}$ = $\bigg|\bigg| \sum\limits_{j=1}^{\infty} \lambda_{j} A_{j} \otimes y_{j} - \sum\limits_{j=1}^{n} \lambda_{j} A_{j} \otimes y_{j}\bigg|\bigg|_{\mathcal{H}\sigma(p)}$
%\newline
%\newline
%\newline
%\hspace*{4.5cm} = $\bigg|\bigg| \sum\limits_{j=n+1}^{\infty} \lambda_{j} A_{j} \otimes y_{j}\bigg|\bigg|_{\mathcal{H}\sigma(p)}$
%\newline
%\newline
%\newline
%\hspace*{4.5cm} = $\inf \bigg\{ ||(\lambda_{j})_{j=n+1}^{\infty}||_{p'} . \sup\limits_{x_{i} \in B_{E_{i}},y' \in B_{F'}} \bigg(\sum\limits_{j=n+1}^{\infty}|A_{j}(x_{1},...,x_{n})y'(y_{j})|^{p})^{\frac{1}{p}} \bigg\} $
%\newline
%\newline
%\newline
%\hspace*{4.5cm} = $\bigg(\sum\limits_{j=n+1}^{\infty} |\lambda_{j}|^{p'}\bigg)^{\frac{1}{p'}} . \sup\limits_{x_{i} \in B_{E_{i}},y' \in B_{F'}} \bigg(\sum\limits_{j=n+1}^{\infty}|A_{j}(x_{1},...,x_{n})y'(y_{j})|^{p}\bigg)^{\frac{1}{p}}$
%\newline
%\newline
%\newline
%\hspace*{4cm} $\leq$ $\bigg(\sum\limits_{j=n+1}^{\infty} |\lambda_{j}|^{p'}\bigg)^{\frac{1}{p'}} . \sup\limits_{x_{i} \in B_{E_{i}},y' \in B_{F'}} \bigg(\sum\limits_{j=1}^{\infty}|A_{j}(x_{1},...,x_{n})y'(y_{j})|^{p}\bigg)^{\frac{1}{p}}$
%\newline
%\newline
%\newline
%\hspace*{4cm} = $K. \bigg(\sum\limits_{j=n+1}^{\infty} |\lambda_{j}|^{p'}\bigg)^{\frac{1}{p'}}$ $\rightarrow 0$, $n \rightarrow \infty$,
\end{proof}

\begin{lemma}\label{llemma} For $1 \leq p < \infty$, $\mathcal{HN}_{\sigma(p)} = \mathcal{N}_{\sigma(p)} \circ {\cal L}$ isometrically. In particular, $\mathcal{HN}_{\sigma(p)}$ is a Banach hyper-ideal of multilinear operators.
\end{lemma}

\begin{proof} Given $ A  \in \mathcal{N}_{\sigma(p)} \circ \mathcal{L}(E_{1},\ldots,E_{n};F)$ and $\varepsilon > 0$, there exist a Banach space $G$, $B \in \mathcal{L}(E_{1},\ldots,E_{n};G)$, an operator $u \in \mathcal{N}_{\sigma(p)}(G;F)$ and a $\sigma(p)$-nuclear representation $\sum\limits_{j=1}^\infty \lambda_jw^*_j \otimes y_j$ of $u$ such that $A = u \circ B$ and
$$\|(\lambda_{j})_{j=1}^{\infty}\|_{p^*} \cdot \sup\limits_{w \in B_{G},y^* \in B_{F^*}} \left( \sum\limits_{j=1}^{\infty} |w^*_{j}(w)y^*(y_{j})|^{p} \right)^{\frac{1}{p}} < (1+ \varepsilon)\|u\|_{{\cal N}_{\sigma(p)}}. $$
We omit the details that $A = \sum\limits_{j=1}^\infty \lambda_j(w^*_j \circ B) \otimes y_j$ is a hyper-$\sigma(p)$-nuclear representation of $A$ and
$$\|A\|_{\mathcal{HN}_{\sigma(p)}}\leq (1 + \varepsilon)\|B\|\cdot \|u\|_{{\cal N}_{\sigma(p)}}, $$
from which it follows that $A$ is hyper-$\sigma(p)$-nuclear and $\|A\|_{\mathcal{HN}_{\sigma(p)}} \leq \|A\|_{\mathcal{N}_{\sigma(p)} \circ {\cal L}}$.

Conversely, let $A \in \mathcal{L}_{\mathcal{HN}\sigma(p)}(E_{1},\ldots,E_{n};F)$ and  $\varepsilon > 0$ be given and let %. Devemos mostrar que $A \in \mathcal{N}_{\sigma(p)} \circ \mathcal{L}(E_{1},\ldots,E_{n};F).$
%\medskip
%\noindent
%Da Propriedade Universal do Produto Tensorial, existe um único operador linear
%$$A_{L} \in \mathcal{L}(E_{1} \widehat{\otimes}_{\pi} .\cdots \widehat{\otimes}_{\pi}E_{n};F)$$
%tal que $A = A_{L} \circ \sigma_{n}$,  Portanto, pela Proposição 3.2 (Paper Botelho/Pilar/Daniel) é suficiente mostrar que $A_{L} \in \mathcal{N}_{\sigma(p)}(E_{1} \widehat{\otimes}_{\pi} \cdots \widehat{\otimes}_{\pi}E_{n};F).$
%\noindent {\bf Lema.} Se $T \in {\cal L}(E_1, \ldots, E_n)$ e $y \in F$, então $(T \otimes y)_L = T_L \otimes y$.5
%\medskip
%\noindent {\bf Dem. do Lema.} Para todos $x_1, \ldots, x_n \in E$,
%\begin{align*}((T_L \otimes y) \circ \sigma_n)(x_1, \ldots, x_n) &= (T_L \otimes y)(x_1 \otimes \cdots \otimes x_n) = T_L(x_1 \otimes \cdots \otimes %x_n)y\\
% &= T(x_1, \ldots, x_n) y = (T \otimes y)(x_1, \ldots, x_n).
%\end{align*}
%Isso prova que $(T_L \otimes y) \circ \sigma_n = T \otimes y$. Mas $T_L \otimes y$ é um operador linear contínuo, então da unicidade da linearização de um %operador multilinear segue que $(T \otimes y)_L = T_L \otimes y$. $\blacksquare$
%\medskip
%\noindent
%Suponhamos que
$A = \sum\limits_{j=1}^{\infty} \lambda_{j}B_{j} \otimes y_{j}$ be a  hyper-$\sigma(p)$-nuclear representation of $A$ such that%, isto é, $(\lambda_{j})_{j=1}^{\infty} \in \ell_{p'}$, $(T_{j})_{j=1}^{\infty} \in \mathcal{L}(E_{1},...,E_{n})$ e $(y_{j})_{j=1}^{\infty} \in F$ e são tais que
%\begin{equation} \label{eq 0}
%\sup\limits_{x_{i} \in B_{E}, y' \in B_{F'}} \bigg( \sum\limits_{j=1}^{\infty} |T_{j}(x_{1},\ldots,x_{n})\textcolor{blue}{y'(y_j)}|^{p} \bigg)^{\frac{1}{p}} < \infty,
%\end{equation}
%\begin{equation} \label{eq 00}\lim\limits_{m \to \infty} \sup\limits_{x_{i} \in B_{E}, y' \in B_{F'}} \bigg( \sum\limits_{j=m}^{\infty} |T_{j}(x_{1},\ldots,x_{n})\textcolor{blue}{y'(y_j)}|^{p} \bigg)^{\frac{1}{p}} = 0 ~{\rm e}
%\end{equation}
\begin{equation}\label{eq 000}\|(\lambda_j)_{j=1}^\infty\|_{p^*}\cdot \sup\limits_{x_{i} \in B_{E}, y^* \in B_{F'}} \bigg( \sum\limits_{j=1}^{\infty} |B_{j}(x_{1},\ldots,x_{n})y^*(y_j)|^{p} \bigg)^{\frac{1}{p}} \leq (1 + \varepsilon)\|A\|_{{\cal HN}\sigma(p)}.
\end{equation}
According to Lemma \ref{lemmel}, the convergence $A = \sum\limits_{j=1}^{\infty} \lambda_{j}B_{j} \otimes y_{j}$ occurs in the usual norm, and since the correspondence %E como a correspondência
\begin{equation*}B \in {\cal L}(E_1, \ldots, E_n;F) \mapsto B_L \in {\cal L}(E_{1} \widehat{\otimes}_{\pi} \cdots \widehat{\otimes}_{\pi}E_{n};F)
\end{equation*}
is an isometric isomoprhism, we have
%é um isomorfismo, logo linear e contínua, temos do Lema
\begin{align*}A_L = \left(\sum\limits_{j=1}^{\infty} \lambda_{j}B_{j} \otimes y_{j} \right)_L = \sum\limits_{j=1}^{\infty} (\lambda_{j}B_{j} \otimes y_{j})_L = \sum_{j=1}^\infty \lambda_{j}(B_{j})_L \otimes y_{j},
  % \left(\lim_m \sum\limits_{j=1}^m \lambda_{j}T_{j} \otimes y_{j} \right)_L = \lim_m \left(\sum\limits_{j=1}^m \lambda_{j}T_{j} \otimes y_{j} \right)_L\\
%&= \lim_m \sum\limits_{j=1}^m \lambda_{j}(T_{j} \otimes y_{j})_L = \lim_m \sum\limits_{j=1}^m \lambda_{j}(T_{j})_L \otimes y_{j} = \sum_{j=1}^\infty \lambda_{j}(T_{j})_L \otimes y_{j},
\end{align*}
in the norm of ${\cal L}(E_{1} \widehat{\otimes}_{\pi} \cdots \widehat{\otimes}_{\pi}E_{n};F)$. Let us prove that $\sum\limits_{j=1}^\infty \lambda_{j}(B_{j})_L \otimes y_{j}$ is a $\sigma(p)$-nuclear representation of $A_L$.
%\medskip
%\noindent
%Dados escalares $\lambda_1, \ldots, \lambda_m$, segue da dualidade $(\ell_p^m)' = \ell_{p'}^m$ e do Teorema de Hahn-Banach, que existem escalares $\theta_1, \ldots, \theta_m$ tais que
%$$\|(\theta_j)_{j=1}^m\|_{p'} = 1 {\rm ~e~} \left(\sum_{j=1}^m|\lambda_j|^p\right)^{1/p} = \left| \sum_{j=1}^m \theta_j \lambda_j \right| .$$
%\noindent{\bf Fato.} Dados escalares $\lambda_1, \ldots, \lambda_m$, existem escalares $\theta_1, \ldots, \theta_m$ tais que
%$$\|(\theta_j)_{j=1}^m\|_{p'} = 1 {\rm ~e~} \left(\sum_{j=1}^m|\lambda_j|^p\right)^{1/p} = \left| \sum_{j=1}^m \theta_j \lambda_j \right| .$$
%\medskip
%\noindent{\bf Dem. do Fato.} Segue da dualidade $(\ell_p^m)' = \ell_{p'}^m$ e do Teorema de Hahn-Banach na forma $\|x\| = \sup_{\|\varphi\|\leq %1}|\varphi(x)|$.
For $k > m$, %using again the isometric correspondence (\ref{isco}), we have %como a norma de um operador linear coincide com a norma da sua linearização, temos
%\begin{equation*}\label{eq 1}\left\| \sum_{j=m}^k T_j \otimes y_j\right\| = \left\|\sum_{j=m}^k (T_j)_L \otimes y_j\right\|.
%\end{equation*}
%Foi visto também que $A = \sum\limits_{j=1}^{\infty} \lambda_{j}T_{j} \otimes y_{j}$ na norma $\|\cdot\|_{H\sigma(p)}$, e portanto a sequência $\left(\sum\limits_{j=1}^n \lambda_j T_j \otimes y_j\right)_{n=1}^\infty$ é convergente na norma $\|\cdot\|_{H\sigma(p)}$, e portanto é uma sequência de Cauchy, isto é,
%\begin{equation}\label{eq 2}
%\lim_{m,n} \left\|\sum_{j= 1}^m \lambda_jT_j \otimes y_j -\sum_{j= 1}^n \lambda_jT_j \otimes y_j\right\|_{H\sigma(p)}=0.
%\end{equation}
%\noindent
  $z \in B_{E_{1} \widehat{\otimes}_{\pi} \cdots \widehat{\otimes}_{\pi}E_{n}}$ and $y^* \in B_{F^*}$, % e cada $j = m, \ldots, k$,
the duality  $(\ell_p^{k-m+1})^* = \ell_{p^*}^{k-m+1}$ and the Hahn-Banach Theorem provide scalars $\theta_{m}, \ldots, \theta_k$ such that $\|(\theta_j)_{j=m}^k\|_{p^*} = 1$ and
 \begin{equation*} \left(\sum_{j=m}^k|(B_j)_L(z)y^*(y_j)|^p\right)^{1/p} = \left| \sum_{j=m}^k \theta_j (B_j)_L(z)y^*(y_j) \right| .
\end{equation*}
%Note que $\|\theta_j y'\| = \|y'\|\leq 1$, e portanto $\theta_j y\in B_{F'}$.
%Usaremos agora o seguinte : (i) o Teorema de Hahn-Banach na forma $\|x\| = \sup_{\|\varphi\|\leq 1}|\varphi(x)|$, (ii) a igualdade (\ref{eq 1}), (iii) as igualdades em (\ref{eq 3}), (iv) a Desigualdade de H\"older para $\frac1p + \frac{1}{p'} = 1$.
It follows that
\begin{align*} \left( \sum_{j=m}^k |(B_j)_L(z) y^*(y_j)|^p \right)^{\frac{1}{p}}& = %\left|\sum_{j=m}^k \theta_j(T_j)_L(z) y^*(y_j)\right| =
 \left|y^*\left(\sum_{j=m}^k \theta_j(B_j)_L(z) y_j\right)\right|\leq %\|y'\|\cdot \left\|\sum_{j=m}^k \theta_j(T_j)_L(z) y_j \right\| \leq
 \left\|\sum_{j=m}^k \theta_j(B_j)_L(z) y_j \right\|\\
& = \left\|\left(\sum_{j=m}^k (B_j)_L \otimes (\theta_jy_j)\right)(z) \right\| %\leq \left\|\sum_{j=m}^k (T_j)_L \otimes (\theta_jy_j) \right\| \cdot \|z\|\\
 \leq \left\|\sum_{j=m}^k (B_j)_L \otimes (\theta_jy_j) \right\| \\&  =\left\|\sum_{j=m}^k \theta_j(B_j)_L \otimes y_j \right\| = \left\|\left(\sum_{j=m}^k \theta_jB_j \otimes y_j \right)_L\right\|\\&= \left\|\sum_{j=m}^k \theta_jB_j \otimes y_j \right\|
%&= \sup_{\|\widetilde{x_i}\|\leq 1} \left\|\left(\sum_{j=m}^k \theta_jT_j \otimes y_j \right)(\widetilde{x_1}, \ldots, \widetilde{x_n}) \right\|\\
 = \sup_{\|\widetilde{x_i}\|\leq 1} \left\|\sum_{j=m}^k \theta_jB_j(\widetilde{x_1}, \ldots, \widetilde{x_n}) y_j  \right\|\\
 &= \sup_{\|\widetilde{x_i}\|\leq 1} \sup_{ \|\widetilde{y^*}\|\leq 1} \left|\widetilde{y^*}\left(\sum_{j=m}^k \theta_jB_j(\widetilde{x_1}, \ldots, \widetilde{x_n}) y_j \right) \right|\\
 & = \sup_{\|\widetilde{x_i}\|\leq 1, \|\widetilde{y^*}\|\leq 1} \left|\sum_{j=m}^k \theta_jB_j(\widetilde{x_1}, \ldots, \widetilde{x_n}) \widetilde{y^*}(y_j) \right|\\
% & \leq \sup_{\|\widetilde{x_i}\|\leq 1, \|\widetilde{y'}\|\leq 1} \sum_{j=m}^k \left| \theta_jT_j(\widetilde{x_1}, \ldots, \widetilde{x_n}) \widetilde{y'}(y_j) \right|\\
 &\leq  \sup_{\|\widetilde{x_i}\|\leq 1, \|\widetilde{y^*}\|\leq 1} \left( \sum_{j=m}^k |\theta_j|^{p^*}\right)^{\frac{1}{p^*}}\cdot \left(\sum_{j=m}^k  |B_j(\widetilde{x_1}, \ldots, \widetilde{x_n}) \widetilde{y^*}(y_j) |^p\right)^{\frac1p}\\&
 =  \sup_{\|\widetilde{x_i}\|\leq 1, \|\widetilde{y^*}\|\leq 1}\left(\sum_{j=m}^k  |B_j(\widetilde{x_1}, \ldots, \widetilde{x_n}) \widetilde{y^*}(y_j) |^p\right)^{\frac1p}.
\end{align*}
Taking the supremum over $\|z\|\leq 1$ and $\|y^*\|\leq 1$ we get
\begin{equation*}\label{eq 5}\sup_{\|z\|\leq 1, \|y^*\|\leq 1}\left( \sum_{j=m}^k |(B_j)_L(z) y^*(y_j)|^p \right)^{\frac{1}{p}} \leq \sup_{\|x_i\|\leq 1, \|y^*\|\leq 1}\left(\sum_{j=m}^k  |B_j(x_1, \ldots, x_n) y^*(y_j) |^p\right)^{\frac1p}
\end{equation*}
whenever $k > m$. Therefore, a usual interchange of suprema argument gives
\begin{align*} \sup_{\|z\|\leq 1, \|y^*\|\leq 1}&\left( \sum_{j=1}^\infty |(B_j)_L(z) y^*(y_j)|^p \right)^{\frac{1}{p}} = %\sup_{\|z\|\leq 1, \|y'\|\leq 1} \sup_k\left( \sum_{j=1}^k |(T_j)_L(z) y'(y_j)|^p \right)^{\frac{1}{p}}\\
   \sup_k \sup_{\|z\|\leq 1, \|y^*\|\leq 1}\left( \sum_{j=1}^k |(B_j)_L(z) y^*(y_j)|^p \right)^{\frac{1}{p}}\\
& \leq  \sup_k \sup_{\|x_i\|\leq 1, \|y^*\|\leq 1}\left(\sum_{j=1}^k  |B_j(x_1, \ldots, x_n) y^*(y_j) |^p\right)^{\frac1p}\\
%& =  \sup_{\|x_i\|\leq 1, \|y'\|\leq 1}\sup_k\left(\sum_{j=1}^k  |T_j(x_1, \ldots, x_n) y'(y_j) |^p\right)^{\frac1p}\\
& = \sup_{\|x_i\|\leq 1, \|y^*\|\leq 1}\left(\sum_{j=1}^\infty  |B_j(x_1, \ldots, x_n) y^*(y_j) |^p\right)^{\frac1p} < \infty,
\end{align*}
and
\begin{align*} \lim_m \sup_{\|z\|\leq 1, \|y^*\|\leq 1}&\left( \sum_{j=m}^\infty |(B_j)_L(z) y^*(y_j)|^p \right)^{\frac{1}{p}} %=  \lim_m\sup_{\|z\|\leq 1, \|y'\|\leq 1} \sup_k\left( \sum_{j=m}^k |(T_j)_L(z) y'(y_j)|^p \right)^{\frac{1}{p}}\\
   =\lim_m  \sup_k \sup_{\|z\|\leq 1, \|y^*\|\leq 1}\left( \sum_{j=m}^k |(B_j)_L(z) y^*(y_j)|^p \right)^{\frac{1}{p}}\\
&  \leq  \lim_m \sup_k \sup_{\|x_i\|\leq 1, \|y^*\|\leq 1}\left(\sum_{j=m}^k  |B_j(x_1, \ldots, x_n) y^*(y_j) |^p\right)^{\frac1p}\\
%& =  \lim_m \sup_{\|x_i\|\leq 1, \|y'\|\leq 1}\sup_k\left(\sum_{j=m}^k  |T_j(x_1, \ldots, x_n) y'(y_j) |^p\right)^{\frac1p}\\
& = \lim_m \sup_{\|x_i\|\leq 1, \|y^*\|\leq 1}\left(\sum_{j=m}^\infty  |B_j(x_1, \ldots, x_n) y^*(y_j) |^p\right)^{\frac1p} =0,
\end{align*}
because $\sum\limits_{j=1}^{\infty} \lambda_{j}B_{j} \otimes y_{j}$ is a  hyper-$\sigma(p)$-nuclear representation of $A$. This proves that $\sum\limits_{j=1}^\infty \lambda_{j}(B_{j})_L \otimes y_{j}$ is a $\sigma(p)$-nuclear representation of $A_L$, hence $A_L$ is a $\sigma(p)$-nuclear linear operator, and
\begin{align*}\|A_L\|_{{\cal N}_{\sigma(p)}}& \leq \|(\lambda_j)_{j=1}^\infty\|_{p^*}\cdot \sup\limits_{\|z\|\leq 1, y^* \in B_{F'}} \bigg( \sum\limits_{j=1}^{\infty} |(B_L)_{j}(z)y^*(y_j)|^{p} \bigg)^{\frac{1}{p}}\\
 & \leq \|(\lambda_j)_{j=1}^\infty\|_{p^*}\cdot \sup\limits_{x_{i} \in B_{E}, y^* \in B_{F^*}} \bigg( \sum\limits_{j=1}^{\infty} |B_{j}(x_{1},\ldots,x_{n})y^*(y_j)|^{p} \bigg)^{\frac{1}{p}}\\& \leq (1 + \varepsilon)\|A\|_{{\cal HN}\sigma(p)}
\end{align*}
for every $\varepsilon > 0$. Once again from \cite[Propositions 3.2 and 3.7]{prims} it follows that $A \in {\cal N}_{\sigma(p)} \circ {\cal L} (E_1, \ldots, E_n;F)$ and $\|A\|_{{\cal N}_{\sigma(p)} \circ {\cal L} } = \|A_L\|_{{\cal N}_{\sigma(p)}} \leq \|A\|_{{\cal HN}\sigma(p)}$.
%. From [Prims, ???] it follows that
%$A \in {\cal N}_{\sigma(p)} \circ {\cal L} (E_1, \ldots, E_n;F)$ and . E, combinando com (\ref{eq 000}), segue também que
%
%para todo $\varepsilon > 0$, onde a primeira igualdade segue da Proposição  Botelho/Pilar/Daniel. Fazendo $\varepsilon \rightarrow 0$, concluímos que $\|A\|_{{\cal N}_{\sigma(p)} \circ {\cal L} } \leq \|A\|_{{\cal H}\sigma(p)}$.
\end{proof}

Our last application represents linear functionals on spaces of hyper-$\sigma(p)$-nuclear multilinear operators as quasi-$\tau(p)$-summing linear operators (cf. Example \ref{exclass}(b)).

\begin{theorem} Let $1 \leq p < \infty$ and suppose that ${\cal L}(E_1, \ldots, E_n)$ has the bounded approximation property. Then
$$[\mathcal{HN}_{\sigma(p)}(E_{1},\ldots,E_{n};F)]^* \stackrel{{\cal B}_n}{=} \Pi_{q\tau(p)}(\mathcal{L}(E_{1},\ldots,E_{n});F^*) $$
for every Banach space $F$.
\end{theorem}

\begin{proof} The linear case of \cite[Theorem 3.4]{ximenaLAA} gives that   $[{\cal N}_{\sigma(p)}(E;F)]^* \stackrel{{\cal B}}{=} \Pi_{q\tau(p)}(E^*,F^*)$ whenever $E^*$ has the bounded approximation property. Since $\Pi_{q\tau(p)}$ is a left semi-operator ideal (Example \ref{exclass}(b)), the result follows from Theorem \ref{tthh} and Lemma \ref{llemma}.
\end{proof}

\medskip

\noindent{\bf Acknowledgments.} The authors thank Ariosvaldo M. Jatob\'a and Ewerton R. Torres for helpful conversations on the subject of this paper.

\bigskip

\noindent Faculdade de Matem\'atica~~~~~~~~~~~~~~~~~~~~~~Instituto de Matem\'atica e Estat\'istica\\
Universidade Federal de Uberl\^andia~~~~~~~~ Universidade de S\~ao Paulo\\
38.400-902 -- Uberl\^andia -- Brazil~~~~~~~~~~~~ 05.508-090  -- S\~ao Paulo -- Brazil\\
e-mail: botelho@ufu.br ~~~~~~~~~~~~~~~~~~~~~~~~~e-mail: raquelwood@ime.usp.br

\end{document}